\documentclass{amsart}
\usepackage{amssymb}
\usepackage{mathrsfs}
\usepackage{cases}
\usepackage{amsmath}
\usepackage{amsfonts}
\usepackage{pifont}
\usepackage{graphicx,amssymb,mathrsfs,amsmath}
\usepackage{tocvsec2}
\newtheorem{thm}{Theorem}[section]
\newtheorem{cor}[thm]{Corollary}
\newtheorem{lem}[thm]{Lemma}
\newtheorem{prop}[thm]{Proposition}
\theoremstyle{definition}
\newtheorem{defn}[thm]{Definition}
\newtheorem{example}[thm]{Example}
\theoremstyle{remark}
\newtheorem{rem}[thm]{Remark}
\numberwithin{equation}{section}

\begin{document}
\title[Multi-dimensional Besicovitch almost periodic...]{Multi-dimensional Besicovitch almost periodic type functions and applications}
\author{M. Kosti\' c}
\address{Faculty of Technical Sciences,
University of Novi Sad,
Trg D. Obradovi\' ca 6, 21125 Novi Sad, Serbia}
\email{marco.s@verat.net}

{\renewcommand{\thefootnote}{} \footnote{
\\ \text{  }  \ \    {\it Key words and phrases.} Besicovitch almost periodic functions in ${\mathbb R}^{n}$, Doss almost periodic functions in ${\mathbb R}^{n}$,
abstract Volterra integro-differential equations.
\\  \text{  }  
The author is partially supported by grant 451-03-68/2020/14/200156 of Ministry
of Science and Technological Development, Republic of Serbia.}}

\begin{abstract}
In this paper, we analyze 
multi-dimensional Besicovitch almost periodic type functions.
We clarify the main structural properties for the introduced classes of Besicovitch almost periodic type  
functions, explore the notion of Besicovitch-Doss almost periodicity in the multi-dimensional setting, and provide some applications of our results to
the abstract Volterra integro-differential equations and the partial differential equations.
\end{abstract}
\maketitle

\section{Introduction and preliminaries}

As is well known, the notion of almost periodicity was introduced by the Danish mathematician H. Bohr around 1924-1926 and later generalized by many other authors (see the research monographs \cite{besik}, \cite{diagana}, \cite{fink}, \cite{gaston}, \cite{nova-mono}, \cite{nova-selected}, \cite{188}, \cite{pankov} and \cite{30} for further information concerning almost periodic functions and their applications).
Suppose that $(X,\| \cdot \|)$ is a complex Banach space and $F : {\mathbb R}^{n} \rightarrow X$ is a continuous function ($n\in {\mathbb N}$). Then it is said that the function $F(\cdot)$ is almost periodic if and only if for each $\epsilon>0$
there exists $l>0$ such that for each ${\bf t}_{0} \in {\mathbb R}^{n}$ there exists ${\bf \tau} \in B({\bf t}_{0},l)\equiv \{ {\bf t} \in {\mathbb R}^{n} : |{\bf t}-{\bf t}_{0}|\leq l\}$ with
\begin{align*}
\bigl\|F({\bf t}+{\bf \tau})-F({\bf t})\bigr\| \leq \epsilon,\quad {\bf t}\in {\mathbb R}^{n};
\end{align*}
here, $|\cdot -\cdot|$ denotes the Euclidean distance in ${\mathbb R}^{n}$ and $\tau$ is usually called an $\epsilon$-almost period of $F(\cdot)$.
Any trigonometric polynomial in ${\mathbb R}^{n}$ is almost periodic, and we know that a continuous function $F(\cdot)$ is almost periodic if and only if there exists a sequence of trigonometric polynomials in ${\mathbb R}^{n}$ which converges uniformly to $F(\cdot).$

If the function $F : {\mathbb R}^{n} \rightarrow X$ is locally $p$-integrable, where $1\leq p<\infty,$
then we say that
$F(\cdot)$ is Stepanov-$p$-almost periodic if and only if for every $\epsilon>0$
there exists $l>0$ such that for each ${\bf t}_{0} \in {\mathbb R}^{n}$ there exists ${\bf \tau} \in B({\bf t}_{0},l) \cap {\mathbb R}^{n}$ with
\begin{align*}
\bigl\|F({\bf t}+{\bf \tau}+{\bf u})-F({\bf t}+{\bf u})\bigr\|_{L^{p}([0,1]^{n} : X)} \leq \epsilon,\quad {\bf t}\in {\mathbb R}^{n}.
\end{align*}
Further on, we say that a locally $p$-integrable function $F : {\mathbb R}^{n} \rightarrow X$ is:
\begin{itemize}
\item[(i)] equi-Weyl-$p$-almost periodic if and only if, for every $\epsilon>0,$ there exist two finite real numbers
$l>0$
and
$L>0$ such that for each ${\bf t}_{0}\in {\mathbb R}^{n}$ there exists $\tau \in B({\bf t}_{0},L)\cap {\mathbb R}^{n}$ with
\begin{align*}
\sup_{{\bf t}\in {\mathbb R}^{n}}\Biggl[l^{-\frac{n}{p}} \bigl\| F({\bf \tau}+\cdot)-F(\cdot) \bigr\|_{L^{p}({\bf t}+l[0,1]^{n} : X)}\Biggr] <\epsilon.
\end{align*}
\item[(ii)] Weyl-$p$-almost periodic if and only if, for every $\epsilon>0,$ there exists a finite real number
$L>0$ such that for each ${\bf t}_{0}\in {\mathbb R}^{n}$ there exists $\tau \in B({\bf t}_{0},L) \cap {\mathbb R}^{n}$ with
\begin{align*}
\limsup_{l\rightarrow +\infty}
\sup_{{\bf t}\in {\mathbb R}^{n}}
\Biggl[l^{-\frac{n}{p}}\bigl\| F({\bf \tau}+\cdot)-F(\cdot) \bigr\|_{L^{p}({\bf t}+l[0,1]^{n}: X)} \Biggr]
<\epsilon.
\end{align*}
\end{itemize}

It is well known that any Bohr almost periodic function is Stepanov-$p$-almost periodic, as well as that any Stepanov-$p$-almost periodic function is equi-Weyl-$p$-almost periodic. The most general class is the class of Weyl-$p$-almost periodic functions which contains all others.
Furthermore, we know that any
equi-Weyl-$p$-almost periodic function is Besicovitch-$p$-almost periodic, and that there exists a Weyl-$p$-almost periodic function $f : {\mathbb R} \rightarrow {\mathbb R}$ which is not Besicovitch-$p$-almost periodic (\cite{nova-selected}). 

The notion of Besicovitch-$p$-almost periodicity for a function $F : {\mathbb R}^{n} \rightarrow X$ can be introduced in many equivalent ways; traditionally, if $F\in L_{loc}^{p}({\mathbb R}^{n} : X),$ then we first define 
$$
\|F\|_{{\mathcal M}^{p}}:=\limsup_{t\rightarrow +\infty}\Biggl[ \frac{1}{(2t)^{n}}\int_{[-t,t]^{n}}\| F({\bf s})\|^{p}\, d{\bf s}\Biggr]^{1/p}.
$$
It can be simply proved that
$
\|\cdot \|_{{\mathcal M}^{p}}$ is a seminorm on the space  ${\mathcal M}^{p} ({\mathbb R}^{n} : X)$ consisting of those $L_{loc}^{p}({\mathbb R}^{n} : X)$-functions $F(\cdot)$ for which $
\|F\|_{{\mathcal M}^{p}}<\infty .$ Denote $K_{p}({\mathbb R}^{n} : X) := \{ f \in {\mathcal M}^{p}({\mathbb R}^{n} : X) \, ; \,
\|F\|_{{\mathcal M}^{p}}= 0\} $
and
$$
M_{p}({\mathbb R}^{n} : X) :={\mathcal M}^{p}({\mathbb R}^{n} : X) / K_{p}({\mathbb R}^{n} : X).
$$
The seminorm $
\|\cdot \|_{{\mathcal M}^{p}}$ on ${\mathcal M}^{p}({\mathbb R}^{n} : X)$ induces the norm $
\|\cdot \|_{M^{p}}$ on $M^{p}({\mathbb R}^{n} : X)$ under which $M^{p}({\mathbb R}^{n} : X)$ is complete; hence, $(M^{p}({\mathbb R}^{n} : X),\|\cdot \|_{M^{p}})$ is a Banach space. It is said that a function $F\in L_{loc}^{p}({\mathbb R}^{n} : X)$ is Besicovitch-$p$-almost periodic if and only if there exists a sequence of
trigonometric polynomials (almost periodic functions, equivalently) converging to $F(\cdot)$ in $(M^{p}({\mathbb R}^{n} : X),\|\cdot \|_{M^{p}}).$
The vector space consisting of all Besicovitch-$p$-almost periodic functions is denoted by $B^{p}({\mathbb R}^{n} : X).$
Clearly, $B^{p}({\mathbb R}^{n} : X)$ is a closed subspace of $M^{p}({\mathbb R}^{n} : X)$ and therefore a Banach space itself. Concerning the Banach space $M^{p}({\mathbb R}^{n} : X)$, we would like to recall that this space is not separable for any finite exponent $p\geq 1;$ see, e.g., \cite[Theorem 18]{pic} which concerns the one-dimensional case.

For further information about 
Besicovitch almost periodic functions, Besicovitch almost automorphic functions and their applications, we refer the reader to
\cite{andreano,andreano1,avant,ayachi,bass,berta,besik1,besik2,folner}, \cite{bruno0,bruno, casado,casado1,fei,grande1,grande2,hilman}, \cite{iann,lenz,y-li,marcin,morsli,panov,piao,pic,sango} and references cited therein;
we would like to specially emphasize here the important research monograph \cite{pankov} by A. A. Pankov. The spatially Besicovitch almost periodic
solutions for certain classes of nonlinear second-order elliptic equations, single higher-order hyperbolic equations and nonlinear
Schr\"odinger equations have been investigated in the fifth chapter of this monograph. For the basic source of information about homogenization in algebras with mean value and generalized Besicovitch spaces (the work of J. L. Woukeng and his coauthors), we refer the reader to \cite[Part II, Chapter 9, pp. 619--621]{nova-selected}.

On the other hand, in \cite{marko-manuel-ap}, A. Ch\'avez et al. have analyzed
various classes of almost periodic type functions of the form $F : \Lambda \times X\rightarrow Y,$ where $(Y,\|\cdot \|_{Y})$ is a complex Banach space and $\emptyset \neq  \Lambda \subseteq {\mathbb R}^{n}$. This research has been continued in \cite{marko-manuel-ap-st} and \cite{fedorov}, where we have analyzed the Stepanov classes and the Weyl classes of multi-dimensional almost periodic functions $F : \Lambda \times X\rightarrow Y.$
Here it worth noticing that
the concept Besicovitch-$p$-almost periodicity has not been well explored for the functions of the form 
$F : \Lambda \rightarrow X,$ where $\emptyset \neq  \Lambda \subseteq {\mathbb R}^{n}$ and $\Lambda\neq {\mathbb R}^{n}$ (some particular results in the one-dimensional setting are given in the monograph \cite{nova-mono}, with $\Lambda =[0,\infty)$).
This fact has strongly influenced us to write this paper,
in which we continue the research studies \cite{marko-manuel-ap,marko-manuel-ap-st,fedorov,c1,doss-rn} by investigating the multi-dimensional Besicovitch almost periodic type functions $F : \Lambda \times X \rightarrow Y,$ where $(Y,\| \cdot \|_{Y})$
is a complex Banach space and $\emptyset \neq \Lambda \subseteq {\mathbb R}^{n}.$ It is worth noting that this is probably the first research article which examines the existence and uniqueness of Besicovitch-$p$-almost periodic solutions for certain classes of PDEs on some proper subdomains of ${\mathbb R}^{n};$ even the most simplest examples of quasi-linear partial differential equations of first order considered here vividly exhibit the necessity of further analyses of Besicovitch-$p$-almost periodic functions which are not defined on the whole Euclidean space ${\mathbb R}^{n}.$

The organization and main ideas
of this paper can be briefly described as follows. In Subsection \ref{karambita}, we recall the basic definitions and results from the theory of Lebesgue spaces with variable exponents
$L^{p(x)}$. The main aim of Section \ref{anatolia} is to introduce and analyze various classes of multi-dimensional Besicovitch almost periodic type functions. We start this section by introducing the class $e-({\mathcal B},\phi,{\rm F})-B^{p(\cdot)}(\Lambda \times X : Y),$ where $\Lambda$ is a general non-empty subset of
${\mathbb R}^{n},$ $p\in {\mathcal P}(\Lambda),$ $\phi : [0,\infty) \rightarrow [0,\infty)$ is Lebesgue measurable and ${\rm F} : (0,\infty) \rightarrow (0,\infty);$
see Definition \ref{ebe}, which is crucial for our further work. The class $PAP_{0,p}(\Lambda,{\mathcal B},{\rm F},\phi)$ of weighted ergodic components, introduced recently in \cite[Definition 6.4.13]{nova-selected}, makes a proper subclass of the class 
$e-({\mathcal B},\phi,{\rm F})-B^{p(\cdot)}(\Lambda \times X : Y)$ since its definition is obtained by plugging the trivial sequence $(P_{k}\equiv 0)$ of trigonometric polynomials in Definition \ref{ebe} (we omit the term ``${\mathcal B}$'' from the notation for the functions of the form $F :\Lambda \rightarrow Y$).
Before proceeding any further, we would like to note that there is a large class of PDEs of first order (second order) whose solutions
are not Besicovitch almost periodic in ${\mathbb R}^{n}$ and which belong to the class $PAP_{0,p}(\Lambda,{\mathcal B},{\rm F},\phi),$ where $\emptyset \neq \Lambda\subsetneq {\mathbb R}^{n}.$  For instance, we have the following:

\begin{example}\label{123}
\begin{itemize}
\item[(i)]
Let $a$ and $c$ be two non-zero real numbers; then the classical $C^{1}$-solution
of the
equation $au_{x}+cu=0$ is given by $u(x,y)=g(y)e^{-(c/a)x},$ $(x,y)\in {\mathbb R}^{2}.$ Keeping in mind the notion introduced after Definition \ref{prespansko}, it can be simply shown that any non-trivial solution of this equation cannot be Besicovitch almost periodic in ${\mathbb R}^{2}$ as well as that the solution always belongs to the class $PAP_{0,1}([0,\infty) \times {\mathbb R},{\rm F},x)$ provided that the function $g(\cdot)$ is Besicovitch bounded and $\lim_{t\rightarrow +\infty}{\rm F}(t)=0$ (see also Example \ref{fat} and Example \ref{maric0} below).
\item[(ii)] The solutions of second-order PDEs on rectangular domains, obtained by the well known method of separation of
variables, can belong to the space $PAP_{0,p}(\Lambda,{\mathcal B},{\rm F},\phi)$. Consider, for the illustration purposes, the heat equation $u_{t}(x,t)=u_{xx}(x,t)$ on the domain $\Lambda =[0,2\pi] \times [0,\infty),$ equipped with the initial conditions $u(x,0)=f(x) \in L^{1}[0,2\pi]$ and $u(0,t)=u(2\pi,t)=0.$ A unique solution of this problem is given by
$$
u(x,t)=\sum_{k=1}^{\infty}b_{k}\sin(kx/2)e^{-\frac{k^{2}}
{4}t},\quad (x,t)\in \Lambda,
$$
where $f(x)=\sum_{k=1}b_{k}\sin(kx/2),$ $x\in [0,2\pi].$ Since there exists a finite real constant $M>0$ such that $|b_{k}|\leq M$ for all $k\in {\mathbb N},$ a simple computation shows that $u(x,t)\in PAP_{0,1}(\Lambda,t^{-\zeta},x)$ for any real number $\zeta>0.$
\end{itemize}
\end{example}

We provide some structural characterizations of class $e-({\mathcal B},\phi,{\rm F})-B^{p(\cdot)}(\Lambda \times X : Y)$ in Proposition \ref{anat} and Theorem \ref{gs} and Proposition \ref{sdffds} (a special attention is paid to the case in which $\phi(x)\equiv x^{\alpha}$ for some $\alpha>0$). A composition principle for multi-dimensional Besicovitch almost periodic functions is clarified in Theorem \ref{krompo}. In Subsection \ref{normalno}, we investigate the notion of
multi-dimensional Besicovitch normality. The class of Besicovitch$-({\mathrm R}, {\mathcal B},\phi,{\rm F})-B^{p(\cdot)}$-normal functions is introduced in Definition
\ref{dfggg} and characterized after that in Proposition \ref{raj-sed}, Proposition \ref{kursk-kursk} and Proposition \ref{sdffds1}.

In Section \ref{novs}, we consider the multi-dimensional analogues of the important research results established by R. Doss in \cite{doss}-\cite{doss1}.
We pay special attention to the analysis of conditions (A), (A)$_{\infty}$, and (AS); see Theorem \ref{paramet}, Proposition \ref{raj} and Proposition \ref{dawq} for some results obtained in this direction. Subsection \ref{bea} is focused on the analysis of  condition (B); the main results obtained in this part are Proposition \ref{zaed} and Proposition \ref{zaed-multi}. We feel it is our duty to emphasize that the above-mentioned results of R. Doss
are primarily intended for the analysis of one-dimensional Besicovitch almost periodic type functions as well as that we have faced ourselves with many serious problems concerning the multi-dimensional extensions of these results. 

Some applications to the abstract Volterra integro-differential equations are furnished in Section \ref{convinv}; it is worth noting that we establish here some new results about the convolution invariance of Besicovitch-$p$-almost periodicity under the actions of infinite convolution products,
and a new result concerning the usually considered convolution invariance of Besicovitch-$p$-almost periodicity. We provide some new applications in the one-dimensional setting, a new application in the analysis of the existence and uniqueness of Besicovitch-$p$-almost periodic solutions of the abstract semilinear fractional Cauchy inclusions and the abstract nonautonomous
differential equations of first order, some new applications to the inhomogeneous heat equation in ${\mathbb R}^{n}$ and evolution systems generated by the bounded perturbations of the Dirichlet Laplacian. 

The final section of paper is reserved for conclusions and final comments about the considered notion. We explain here that the notion of Besicovitch-$(p,c)$-almost periodicity, where $c\in {\mathbb C}\setminus \{0\},$ cannot be so simply introduced and analyzed. Concerning some open problems proposed, we would like to notice that we introduce the notion of admissibility  with respect to the class ${\mathcal C}_{\Lambda}$ of Besicovitch almost periodic type functions in Definition \ref{bes-ext}, and propose  after that an open problem concerning the extensions of Besicovitch almost periodic type functions to the whole Euclidean space ${\mathbb R}^{n}$; a similar question has recently been posed for the class of (equi-)Weyl almost periodic type functions in \cite{nova-selected}. In addition to the above, we provide numerous illustrative examples and remarks about the multi-dimensional Besicovitch almost periodic type functions under our consideration.

We use the standard notation throughout the paper.
We assume that $(X,\| \cdot \|)$, $(Y, \|\cdot\|_Y)$ and $(Z, \|\cdot\|_Z)$ are complex Banach spaces, $n\in {\mathbb N},$ $\emptyset  \neq \Lambda \subseteq {\mathbb R}^{n},$
${\mathcal B}$ is a non-empty collection of non-empty subsets of $X,$ and ${\mathrm R}$ is a non-empty collection of sequences in ${\mathbb R}^{n}.$ We assume that
for each $x\in X$ there exists $B\in {\mathcal B}$ such that $x\in B;$ $(e_{1},e_{2},...,e_{n})$ denotes the standard basis of ${\mathbb R}^{n}.$ By
$L(X,Y)$ we denote the Banach algebra of all bounded linear operators from $X$ into
$Y;$ $L(X,X)\equiv L(X)$. If $A: D(A) \subseteq X \mapsto X$ is a closed linear operator,
then its range will be denoted by
$R(A)$; ${\rm I}$ stands for the identity operator on $Y.$ 
Set 
${\mathbb N}_{n}:=\{1,..., n\}$ and $\Delta_{n}:=\{ (t,t,...,t)  \in {\mathbb R}^{n} : t\in {\mathbb R} \}$. By $C_{b}({\mathbb R}^{n} : X)$ we denote the Banach space of all bounded continuous functions $F: {\mathbb R}^{n} \rightarrow X$ equipped with the sup-norm. By a convex polyhedral in ${\mathbb R}^{n},$ we mean any set $\Lambda$ of the form
$$
\Lambda=\bigl\{ \alpha_{1} {\bf v}_{1} +\cdot \cdot \cdot +\alpha_{n}{\bf v}_{n}  : \alpha_{i} \geq 0\mbox{ for all }i\in {\mathbb N}_{n} \bigr\},
$$ 
where $(v_{1},v_{2},...,v_{n})$ is a basis of ${\mathbb R}^{n}.$

For the sequel, we need the following result which can be deduced in 
almost the same way as in the proof of \cite[Proposition 2]{doss}:

\begin{lem}\label{lemara}
Suppose that the function $F: {\mathbb R}^{n}\rightarrow X$ is Bohr almost periodic. If for every $a_{1}\neq 0,...,a_{n}\neq 0$ we have ($a=( a_{1},...,a_{n})$):
$$
\lim_{k\rightarrow +\infty}\frac{1}{k}\sum_{j=0}^{k-1}F({\bf t}+(k-1)a)=0,
$$
uniformly in ${\bf t}\in {\mathbb R}^{n},$ then $F\equiv 0.$
\end{lem}

Although it could be of some importance, we will not discuss here the question whether the statement of Lemma \ref{lemara} can be extended to the almost automorphic functions (the uniformly recurrent functions).

We also need the following notion. Suppose that $\emptyset  \neq \Lambda' \subseteq {\mathbb R}^{n},$ 
$\emptyset \neq  \Lambda \subseteq {\mathbb R}^{n}$,
$F : \Lambda \times X \rightarrow Y$ is a continuous function and $\Lambda +\Lambda' \subseteq \Lambda.$ Then we say that
$F(\cdot ; \cdot)$ is Bohr $({\mathcal B},\Lambda')$-almost periodic (Bohr ${\mathcal B}$-almost periodic, if $\Lambda'=\Lambda$) if and only if for every $\epsilon>0$ and $B\in {\mathcal B}$
there exists $l>0$ such that for each ${\bf t}_{0} \in \Lambda'$ there exists ${\bf \tau} \in B({\bf t}_{0},l) \cap \Lambda'$ with
\begin{align*}
\bigl\|F({\bf t}+{\bf \tau};x)-F({\bf t};x)\bigr\|_{Y} \leq \epsilon,\quad {\bf t}\in \Lambda,\ x\in B.
\end{align*}
Furthermore, we say that the function $F(\cdot ; \cdot)$ is  $({\mathcal B},\Lambda')$-uniformly recurrent (${\mathcal B}$-uniformly recurrent, if $\Lambda'=\Lambda$) if and only if, for every $B\in {\mathcal B},$
there exists a sequence $({\bf \tau}_{n})$ in $\Lambda'$ such that $\lim_{k\rightarrow +\infty} |{\bf \tau}_{k}|=+\infty$ and
$$
\lim_{k\rightarrow +\infty}\sup_{{\bf t}\in \Lambda , x\in B} \bigl\|F({\bf t}+{\bf \tau}_{k};x)-F({\bf t};x)\bigr\|_{Y} =0.
$$
If we consider the functions of the form $F : \Lambda \rightarrow Y,$ then we omit the term ``${\mathcal B}$'' from the notation.

For further information concerning multi-dimensional almost periodic type functions, multi-dimensional almost automorphic type functions and their applications, we refer the reader to our newly published research monograph \cite{nova-selected}. We have been forced to quote this monograph multiple times henceforth since some recent results of ours concerning multi-dimensional almost periodic type functions and their generalizations are still not accepted or published in the final form.

\subsection{Lebesgue spaces with variable exponents
$L^{p(x)}$}\label{karambita}

Let $\emptyset \neq \Omega \subseteq {\mathbb R}^{n}$ be a nonempty Lebesgue measurable subset and let
$M(\Omega  : X)$ denote the collection of all measurable functions $f: \Omega \rightarrow X;$ $M(\Omega):=M(\Omega : {\mathbb R}).$ Further on, by ${\mathcal P}(\Omega)$ we denote the vector space of all Lebesgue measurable functions $p : \Omega \rightarrow [1,\infty].$
For any $p\in {\mathcal P}(\Omega)$ and $f\in M(\Omega : X),$ we set
$$
\varphi_{p(x)}(t):=\left\{
\begin{array}{l}
t^{p(x)},\quad t\geq 0,\ \ 1\leq p(x)<\infty,\\ \\
0,\quad 0\leq t\leq 1,\ \ p(x)=\infty,\\ \\
\infty,\quad t>1,\ \ p(x)=\infty
\end{array}
\right.
$$
and
$$
\rho(f):=\int_{\Omega}\varphi_{p(x)}(\|f(x)\|)\, dx .
$$
We define the Lebesgue space
$L^{p(x)}(\Omega : X)$ with variable exponent
by
$$
L^{p(x)}(\Omega : X):=\Bigl\{f\in M(\Omega : X): \lim_{\lambda \rightarrow 0+}\rho(\lambda f)=0\Bigr\}.
$$
It is well known that
\begin{align*}
L^{p(x)}(\Omega : X)=\Bigl\{f\in M(\Omega : X):  \mbox{ there exists }\lambda>0\mbox{ such that }\rho(\lambda f)<\infty\Bigr\};
\end{align*}
see, e.g., \cite[p. 73]{variable}.
For every $u\in L^{p(x)}(\Omega : X),$ we introduce the Luxemburg norm of $u(\cdot)$ by
$$
\|u\|_{p(x)}:=\|u\|_{L^{p(x)}(\Omega :X)}:=\inf\Bigl\{ \lambda>0 : \rho(u/\lambda)    \leq 1\Bigr\}.
$$
Equipped with this norm, $
L^{p(x)}(\Omega : X)$ becomes a Banach space (see e.g. \cite[Theorem 3.2.7]{variable} for the scalar-valued case), coinciding with the usual Lebesgue space $L^{p}(\Omega : X)$ in the case that $p(x)=p\geq 1$ is a constant function.
Further on, for any $p\in M(\Omega),$ we define
$$
p^{-}:=\text{essinf}_{x\in \Omega}p(x) \ \ \mbox{ and } \ \ p^{+}:=\text{esssup}_{x\in \Omega}p(x).
$$
Set
$$
D_{+}(\Omega ):=\bigl\{ p\in M(\Omega): 1 \leq p^{-}\leq p(x) \leq p^{+} <\infty \mbox{ for a.e. }x\in \Omega \bigr \}.
$$
If $p\in D_{+}(\Omega),$ then we know
$$
L^{p(x)}(\Omega : X)=\Bigl\{f\in M(\Omega : X)  \, ; \,  \mbox{ for all }\lambda>0\mbox{ we have }\rho(\lambda f)<\infty\Bigr\}.
$$

We will use the following lemma (cf. \cite{variable} for the scalar-valued case):

\begin{lem}\label{aux}
\begin{itemize}
\item[(i)] (The H\"older inequality) Let $p,\ q,\ r \in {\mathcal P}(\Omega)$ such that
$$
\frac{1}{q(x)}=\frac{1}{p(x)}+\frac{1}{r(x)},\quad x\in \Omega .
$$
Then, for every $u\in L^{p(x)}(\Omega : X)$ and $v\in L^{r(x)}(\Omega),$ we have $uv\in L^{q(x)}(\Omega : X)$
and
\begin{align*}
\|uv\|_{q(x)}\leq 2 \|u\|_{p(x)}\|v\|_{r(x)}.
\end{align*}
\item[(ii)] Let $\Omega $ be of a finite Lebesgue's measure and let $p,\ q \in {\mathcal P}(\Omega)$ such $q\leq p$ a.e. on $\Omega.$ Then
 $L^{p(x)}(\Omega : X)$ is continuously embedded in $L^{q(x)}(\Omega : X),$ and the constant of embedding is less than or equal to
$2(1+m(\Omega)).$
\item[(iii)] Let $f\in L^{p(x)}(\Omega : X),$ $g\in M(\Omega : X)$ and $0\leq \|g\| \leq \|f\|$ a.e. on $\Omega .$ Then $g\in L^{p(x)}(\Omega : X)$ and $\|g\|_{p(x)}\leq \|f\|_{p(x)}.$
\end{itemize}
\end{lem}

For further information concerning the Lebesgue spaces with variable exponents
$L^{p(x)},$ we refer the reader to the monograph \cite{variable} by L. Diening et al., and the list of references quoted therein.

\section{Multi-dimensional Besicovitch almost periodic type functions}\label{anatolia}

Suppose that
$\Lambda$ is a general non-empty subset of
${\mathbb R}^{n}$ as well as that $p\in {\mathcal P}(\Lambda)$, the function $\phi : [0,\infty) \rightarrow [0,\infty)$ is Lebesgue measurable and ${\rm F} : (0,\infty) \rightarrow (0,\infty).$
Set
$$
\Lambda'':=\bigl\{ \tau \in {\mathbb R}^{n} : \tau +\Lambda \subseteq \Lambda \bigr\}.
$$
Unless stated otherwise, we assume henceforth that 
$\emptyset \neq \Omega \subseteq {\mathbb R}^{n}$ is a compact set with positive Lebesgue measure, as well as that $\Lambda +l\Omega \subseteq \Lambda$ for all $l>0,$ 
and
$\emptyset \neq \Lambda' \subseteq \Lambda'',$ i.e.,
$\Lambda +\Lambda' \subseteq \Lambda .$
In this section, we investigate the multi-dimensional Besicovitch almost periodic type functions, paying a special attention to 
the class $e-({\mathcal B},\phi,{\rm F})-B^{p(\cdot)}(\Lambda \times X : Y)$ and the class of Besicovitch$-({\mathrm R}, {\mathcal B},\phi,{\rm F})-B^{p(\cdot)}$-normal functions.

Recall, a trigonometric polynomial $P: \Lambda \times X\rightarrow Y$ is any linear combination of functions like $({\bf t};x)\mapsto e^{i\langle \lambda, {\bf t}\rangle}c(x),$ where $c: X\rightarrow Y$ is a continuous function; a continuous function $F: \Lambda \times X\rightarrow Y$ is said to be strongly ${\mathcal B}$-almost periodic if and only if for every $B\in {\mathcal B}$ we can find a sequence $(P_{k}^{B}(\cdot;\cdot))_{k\in {\mathbb N}}$ of trigonometric polynomials which converges to $F(\cdot;\cdot)$, uniformly on $\Lambda \times B.$ We omit the term ``${\mathcal B}$'' from the notation if $X=\{0\}.$

We are ready to
introduce the following notion:

\begin{defn}\label{ebe}
Suppose that $F: \Lambda \times X \rightarrow Y,$ $\phi : [0,\infty) \rightarrow [0,\infty)$ and ${\rm F} : (0,\infty) \rightarrow (0,\infty).$ Then we say that the function $F(\cdot;\cdot)$ belongs to the class $e-({\mathcal B},\phi,{\rm F})-B^{p(\cdot)}(\Lambda \times X : Y)$ if and only if for each set $B\in {\mathcal B}$ there exists a sequence $(P_{k}(\cdot;\cdot))$ of trigonometric polynomials such that
\begin{align}\label{jew}
\lim_{k\rightarrow +\infty}\limsup_{t\rightarrow +\infty}{\rm F}(t)\sup_{x\in B}\Bigl[\phi\Bigl( \bigl\|F({\bf t};x)-P_{k}({\bf t};x)\bigr\|_{Y}\Bigr) \Bigr]_{L^{p({\bf t})}(\Lambda_{t} )}=0.
\end{align}
If $\phi(x)\equiv x,$ then we omit the term ``$\phi$'' from the notation; if $X=\{0\},$ then we omit the term ``${\mathcal B}$'' from the notation.
\end{defn}

Immediately from definition, it follows that, for every $F\in e-({\mathcal B},{\rm F})-B^{p(\cdot)}(\Lambda \times X : Y)$ and $\lambda \in {\mathbb R}^{n},$ we have $e^{i \langle \lambda , \cdot \rangle}F\in e-({\mathcal B},{\rm F})-B^{p(\cdot)}(\Lambda \times X : Y).$
The Weyl class $e-{\mathcal B}-W^{p}_{\Omega}(\Lambda \times X : Y),$ introduced in \cite[Definition 6.3.18]{nova-selected} with $p(\cdot)\equiv p\in [1,\infty)$ and ${\rm F}(t)\equiv t^{-n/p},$ makes a subclass of the class $e-({\mathcal B},{\rm F})-B^{p(\cdot)}(\Lambda \times X : Y),$ provided some reasonable choices of compact set $\Omega;$ for example, we have that $e-{\mathcal B}-W^{p}_{\Omega}(\Lambda \times X : Y)\subseteq e-({\mathcal B},{\rm F})-B^{p(\cdot)}(\Lambda \times X : Y)$ if $\Omega=[0,1]^{n},$ and $\Lambda =[0,\infty)^{n}$ or $\Lambda={\mathbb R}^{n}.$ 

Moreover, we have the following:
\begin{itemize}
\item[(i)] Equipped with the usual operations, the set $e-({\mathcal B},\phi,{\rm F})-B^{p(\cdot)}(\Lambda \times X : Y)$ forms a vector space provided that the function $\phi(\cdot)$ is monotonically increasing and there exists a finite real constant $c>0$ such that $\phi(x+y)\leq c[\phi(x)+\phi(y)]$ for all $x,\ y\geq 0.$
\item[(ii)] For every $\tau \in \Lambda'',$ $x_{0}\in X$ and $F\in e-({\mathcal B},{\rm F})-B^{p(\cdot)}(\Lambda \times X : Y)$, we have $F(\cdot +\tau;\cdot+x_{0})\in e-({\mathcal B}_{x_{0}},{\rm F})-B^{p(\cdot)}(\Lambda \times X : Y)$ with ${\mathcal B}_{x_{0}}\equiv \{-x_{0}+B : B\in {\mathcal B}\},$ provided that $p(\cdot)\equiv p\in [1,\infty)$ and there exist
two finite real constants $c_{\tau}>0$ and $t_{\tau}>0$ such that ${\rm F}(t)\leq c_{\tau}{\rm F}(t+|\tau|),$ $t\geq t_{\tau}.$
\item[(iii)] Suppose that the function $\phi(\cdot)$ is monotonically increasing, continuous at the point zero, there exists a finite real constant $c>0$ such that $\phi(x+y)\leq c[\phi(x)+\phi(y)]$ for all $x,\ y\geq 0,$ and the mapping $t\mapsto {\rm F}(t)[1]_{L^{p(\cdot)}(\Lambda_{t})},$ $t>0$ is bounded at plus infinity. Then $F\in e-({\mathcal B},\phi,{\rm F})-B^{p(\cdot)}(\Lambda \times X : Y)$ if and only if for each set $B\in {\mathcal B}$ there exists a sequence $(F_{k}(\cdot;\cdot))$ of strongly ${\mathcal B}$-almost periodic functions such that \eqref{jew} holds with the polynomial $P_{k}(\cdot;\cdot)$ replaced therein with the function $F_{k}(\cdot;\cdot).$
\item[(iv)] Let the assumptions of (iii) hold and let there exist a function $\varphi : [0,\infty) \rightarrow [0,\infty)$ such that $\phi(xy)\leq \varphi(x)\phi(y)$ for all $x,\ y\geq 0.$ Suppose that $h : Y\rightarrow Z$ is a Lipschitz continuous function and $F\in e-(\phi,{\rm F})-B^{p(\cdot)}(\Lambda : Y).$ Using \cite[Proposition 6.1.11]{nova-selected} and the fact that any uniformly continuous Bohr almost periodic function $F: \Lambda \rightarrow Y,$ where $\Lambda$ is a convex polyhedral in ${\mathbb R}^{n},$ is strongly almost periodic, we can prove that $h\circ F \in e-(\phi,{\rm F})-B^{p(\cdot)}(\Lambda : Z).$  
\end{itemize}

Concerning the notion introduced in Definition \ref{ebe}, it is clear that the use of constant coefficients $p(\cdot) \equiv p\in [1,\infty)$ is unquestionably the
best. On the other hand, in the introductory part of \cite{nova-selected}, we have  emphasized that,  from the theoretical
point of view, the use of constant coefficients is not adequately enough because many structural results from the theory of generalized almost periodic functions can be further extended using some results from the theory of Lebesgue spaces with variable exponents $L^{p(x)}.$ Further on, it is clear that Definition \ref{ebe} covers some cases that can be freely called patological; for example, case in which $p\notin D_{+}(\Lambda)$ can be considered:

\begin{example}\label{fat}
Suppose that $\Lambda=[0,\infty),$ and $F : \Lambda \rightarrow {\mathbb R}$ is given by $F(t):=1,$ if there exists $j\in {\mathbb N}\setminus \{1\}$ such that $t\in [j^{2}-1,j^{2}],$ and
$F(t):=0,$ otherwise. Let $p(x)\equiv 1+x^{2},$ $\phi(x)\equiv x$ and
$$
\limsup_{t\rightarrow +\infty}{\rm F}(t) \cdot \inf\Biggl\{ \lambda>0 : \sum_{2\leq j\leq \sqrt{t}} \lambda^{-j^{4}+2j^{2}-2}\leq 1 \Biggr\}=0.
$$   
Then a simple computation with the Luxemburg norm shows that \eqref{jew} is satisfied with the trivial sequence $(P_{k}\equiv 0)$ of trigonometric polynomials, so that $F\in 
PAP_{0,p}(\Lambda,{\rm F},\phi)\subseteq
e-(\phi,{\rm F})-B^{p(\cdot)}(\Lambda  : {\mathbb C}).$
\end{example}

Let $\rho$ be a binary relation on $Y.$ For the sequel, we
need the following notion: 

\begin{defn}\label{prespansko} (see \cite[Definition 1]{doss-rn})
\begin{itemize}
\item[(i)] Suppose that the function $F : \Lambda \times X \rightarrow Y$ satisfies that $\phi(\| F(\cdot;x)\|_{Y})\in L^{p(\cdot)}(\Lambda_{t})$ for all $t>0$ and $x\in X.$ Then we say that the function $F(\cdot;\cdot)$ is Besicovitch-$(p,\phi,{\rm F},{\mathcal B})$-bounded if and only if, for every $B\in {\mathcal B},$ there exists a finite real number $M_{B}>0$ such that
\begin{align*}
\limsup_{t\rightarrow +\infty}{\rm F}(t)\sup_{x\in B}\Bigl[\phi \bigl(\| F(\cdot ;x)\|_{Y}\bigr)\Bigr]_{L^{p(\cdot)}(\Lambda_{t})} \leq M_{B}.
\end{align*}
\item[(ii)] Suppose that the function $F : \Lambda \times X \rightarrow Y$ satisfies that $\phi(\| F(\cdot +\tau ;x)-y_{\cdot;x}\|_{Y})\in L^{p(\cdot)}(\Lambda_{t})$ for all $t>0,$ $x\in X,$ $\tau \in \Lambda'$
and $y_{\cdot;x}\in \rho(F(\cdot;x)).$
\begin{itemize}
\item[(a)] We say that the function $F : \Lambda \times X \rightarrow Y$ is Besicovitch-$(p,\phi,{\rm F},{\mathcal B},\Lambda',\rho)$-continuous if and only if, for every $B\in {\mathcal B}$ as well as for every $t>0,$ $x\in B$ and $\cdot \in \Lambda_{t},$ we have the existence of an element $y_{\cdot;x}\in \rho (F(\cdot;x))$ such that
\begin{align*}
\lim_{\tau \rightarrow 0,\tau \in \Lambda'}\limsup_{t\rightarrow +\infty}{\rm F}(t)\sup_{x\in B}\Bigl[\phi \bigl(\| F(\cdot +\tau ;x)-y_{\cdot;x}\|_{Y}\bigr)\Bigr]_{L^{p(\cdot)}(\Lambda_{t})}=0.
\end{align*}
\item[(b)] We say that the function $F(\cdot;\cdot)$ is Doss-$(p,\phi,{\rm F},{\mathcal B},\Lambda',\rho)$-almost periodic if and only if,
for every $B\in {\mathcal B}$ and $\epsilon>0,$ there exists $l>0$ such that for each ${\bf t}_{0}\in \Lambda'$ there exists a point $\tau \in B({\bf t}_{0},l) \cap \Lambda'$ such that, for every $t>0,$ $x\in B$ and $\cdot \in \Lambda_{t},$ we have the existence of an element $y_{\cdot;x}\in \rho (F(\cdot;x))$ such that
\begin{align*}
\limsup_{t\rightarrow +\infty}{\rm F}(t)\sup_{x\in B}\Bigl[\phi \bigl(\| F(\cdot +\tau ;x)-y_{\cdot;x}\|_{Y}\bigr)\Bigr]_{L^{p(\cdot)}(\Lambda_{t})}<\epsilon.
\end{align*}
\item[(c)] We say that the function $F(\cdot;\cdot)$ is  Doss-$(p,\phi,{\rm F},{\mathcal B},\Lambda',\rho)$-uniformly recurrent if and only if,
for every $B\in {\mathcal B},$ there exists a sequence $(\tau_{k})\in \Lambda'$ such that, for every $t>0,$ $x\in B$ and $\cdot \in \Lambda_{t},$ we have the existence of an element $y_{\cdot;x}\in \rho (F(\cdot;x))$ such that
\begin{align*}
\lim_{k\rightarrow +\infty}\limsup_{t\rightarrow +\infty}{\rm F}(t)\sup_{x\in B}\Bigl[\phi \bigl(\| F(\cdot +\tau_{k} ;x)-y_{\cdot;x}\|_{Y}\bigr)\Bigr]_{L^{p(\cdot)}(\Lambda_{t})}=0.
\end{align*}
\end{itemize}
\end{itemize}
\end{defn}

As before, we omit the term ``${\mathcal B}$'' if $X=\{0\},$ the term ``$\Lambda'$'' if $\Lambda'=\Lambda,$ and the term ``$\rho$'' if $\rho ={\mathrm I}.$
The usual notion of Besicovitch-$p$-almost periodicity (Doss-$p$-almost periodicity) for the function $F : \Lambda \rightarrow Y,$ where $1\leq p<+\infty,$ is obtained by plugging $\phi(x)\equiv x$ and ${\rm F}(t)\equiv t^{-n/p}$ ($\phi(x)\equiv x$, ${\rm F}(t)\equiv t^{-n/p}$ and   $\Lambda'=\Lambda,$ $\rho ={\mathrm I}$). Further on, we say that a function $F: \Lambda \rightarrow Y$ is Besicovitch almost periodic (Doss almost periodic) if and only if $F(\cdot)$ is Besicovitch-$1$-almost periodic (Doss-$1$-almost periodic).
Let us recall that, in the usual setting, a Doss almost periodic function $f : {\mathbb R}\rightarrow {\mathbb C}$ is not generally Besicovitch almost periodic; for example, A. N. Dabboucy and H. W. Davies have constructed an example of such a function which has the mean value equal to zero (cf. \cite[pp. 352-354]{cifre} for more details). 

Now we will state and prove the following result:

\begin{prop}\label{anat}
Suppose that ${\mathcal B}$ consists of bounded subsets of $X$, $F : \Lambda \times X \rightarrow Y$ and, for every fixed element $x\in X,$ the function $F(\cdot;x)$ is Lebesgue measurable. 
\begin{itemize}
\item[(i)]
Suppose that the function $\phi(\cdot)$ is monotonically increasing. If there exists a finite real constant $t_{0}>0$ such that ${\rm F}(t)\leq [\|1\|_{L^{p(\cdot)}(\Lambda_{t})}]^{-1},$ $t\geq t_{0},$ then any function $F\in e-({\mathcal B},\phi,{\rm F})-B^{p(\cdot)}(\Lambda \times X : Y)$ is Besicovitch-$(p,\phi,{\rm F},{\mathcal B})$-bounded. 
\item[(ii)] Suppose that 
$\phi(\cdot)$ is monotonically increasing and continuous at the point $t=0$.
Let $p(\cdot)\equiv p\in [1,\infty)$, and let there exist finite real constants $c>0$ and $t_{0}>0$ such that, for every $t\geq t_{0},$ we have ${\rm F}(t+1)\geq c{\rm F}(t)$ and ${\rm F}(t)\leq [m(\Lambda_{t})]^{-(1/p)}.$ Then any function $F\in e-({\mathcal B},\phi,{\rm F})-B^{p}(\Lambda \times X : Y)$ is Besicovitch-$(p,\phi,{\rm F},{\mathcal B},\Lambda',{\rm I})$-continuous for any set $\Lambda' \subseteq \Lambda''. $
\item[(iii)] Suppose that $p(\cdot)\equiv p\in [1,\infty),$ $\phi(\cdot)$ has the same properties as in \emph{(ii)}, as well as that for every real number $a>0$ there exist finite real constants $c_{a}>0$ and $t_{a}>0$ such that, for every $t\geq t_{a},$ we have ${\rm F}(t+a)\geq c_{a}{\rm F}(t)$ and ${\rm F}(t)\leq [m(\Lambda_{t})]^{-(1/p)}.$
Let
$F\in e-({\mathcal B},\phi,{\rm F})-B^{p}(\Lambda \times X : Y).$
Then the following holds:
\begin{itemize}
\item[(a)]
The function $F(\cdot;\cdot)$ is Doss-$(p,\phi,{\rm F},{\mathcal B},\Lambda,{\rm I})$-almost periodic, provided that $\Lambda+\Lambda \subseteq \Lambda$ and, for every points $( t_{1},..., t_{n} )\in  \Lambda$ and $( \tau_{1},..., \tau_{n} )\in  \Lambda,$ 
the points $( t_{1},t_{2}+\tau_{2},..., t_{n}+\tau_{n} ),$ $( t_{1},t_{2},t_{3}+\tau_{3},..., t_{n}+\tau_{n} ),...,$ $( t_{1},t_{2},..., t_{n-1}, t_{n}+\tau_{n} ),$ also belong to $ \Lambda .$
\item[(b)] The function $F(\cdot;\cdot)$ is Doss-$(p,\phi,{\rm F},{\mathcal B},\Lambda \cap \Delta_{n},{\rm I})$-almost periodic, provided that 
$\Lambda \cap \Delta_{n}\neq \emptyset,$ $\Lambda+(\Lambda\cap \Delta_{n})\subseteq \Lambda$ and that, for every points $( t_{1},..., t_{n} )\in \Lambda$ and $( \tau,..., \tau )\in \Lambda\cap \Delta_{n},$ 
the points $( t_{1},t_{2}+\tau,..., t_{n}+\tau ),$ $( t_{1},t_{2},t_{3}+\tau,..., t_{n}+\tau ),... ,$ $( t_{1},t_{2},..., t_{n-1}, t_{n}+\tau ),$ also belong to $\Lambda\cap \Delta_{n}.$
\end{itemize}
\end{itemize}
\end{prop}

\begin{proof}
In order to prove (i), fix a set $B\in {\mathcal B}.$ Then $B$ is bounded and we have the existence of a trigonometric polynomial $P_{k}(\cdot;\cdot)$ such that
\begin{align*}
\sup_{x\in B}&\Bigl[\phi \bigl(\| F(\cdot ;x)\|_{Y}\bigr)\Bigr]_{L^{p(\cdot)}(\Lambda_{t})}
\\& \leq \sup_{x\in B}\Bigl[\phi \bigl(\| F(\cdot ;x)-P_{k}(\cdot;x)\|_{Y}\bigr)\Bigr]_{L^{p(\cdot)}(\Lambda_{t})} +
\sup_{x\in B}\Bigl[\phi \bigl(\| P_{k}(\cdot;x)\|_{Y}\bigr)\Bigr]_{L^{p(\cdot)}(\Lambda_{t})}
\\& \leq (\epsilon/2{\rm F}(t))+\sup_{x\in B}\Bigl[\phi \bigl(\| P_{k}(\cdot;x)\|_{Y}\bigr)\Bigr]_{L^{p(\cdot)}(\Lambda_{t})}.
\end{align*}
Let $P_{k}(\cdot;x)=\sum_{j=0}^{l}e^{i\langle \lambda_{l} , \cdot \rangle}c_{l}(x)$ for some integer $l\in {\mathbb N},$ points $\lambda_{1},...,\lambda_{l}$ from ${\mathbb R}^{n}$ and continuous functions $c_{1}(\cdot),...,c_{l}(\cdot)$ from $X$ into $Y.$ Then we have the existence of a finite real constant $c_{B}>0$ such that (see also Lemma \ref{aux}(ii)):
\begin{align*}
\sup_{x\in B}&\Bigl[\phi \bigl(\| P_{k}(\cdot;x)\|_{Y}\bigr)\Bigr]_{L^{p(\cdot)}(\Lambda_{t})}\leq \sup_{x\in B}\Biggl[\phi \Biggl(\Biggl\| 
\sum_{j=0}^{l}e^{i\langle \lambda_{l} , \cdot \rangle}c_{l}(x)\Biggr\|_{Y}\Biggr)\Biggr]_{L^{p(\cdot)}(\Lambda_{t})}
\\& \leq \sup_{x\in B}\Biggl[\phi \Biggl(
\sum_{j=0}^{l}\bigl\|c_{l}(x)\bigr\|_{Y}\Biggr)\Biggr]_{L^{p(\cdot)}(\Lambda_{t})}\leq \Bigl[\phi \bigl( c_{B} \bigr)\Bigr]_{L^{p(\cdot)}(\Lambda_{t})}\leq \phi \bigl( c_{B} \bigr){\rm F}(t)^{-1},\quad t\geq t_{0}.
\end{align*}
The proof of (ii) is quite similar and follows from the decomposition:
 \begin{align*}
\sup_{x\in B}& \Bigl[\phi \bigl(\| F(\cdot +\tau ;x)-F(\cdot  ;x)\|_{Y}\bigr)\Bigr]_{L^{p(\cdot)}(\Lambda_{t})}
\\& \leq \sup_{x\in B} \Bigl[\phi \bigl(\| F(\cdot +\tau ;x)-P_{k}(\cdot +\tau ;x)\|_{Y}\bigr)\Bigr]_{L^{p(\cdot)}(\Lambda_{t})}
\\&+\sup_{x\in B} \Bigl[\phi \bigl(\| P_{k}(\cdot +\tau ;x)-P_{k}(\cdot  ;x)\|_{Y}\bigr)\Bigr]_{L^{p(\cdot)}(\Lambda_{t})}
\\&+\sup_{x\in B} \Bigl[\phi \bigl(\| P_{k}(\cdot  ;x)-F(\cdot  ;x)\|_{Y}\bigr)\Bigr]_{L^{p(\cdot)}(\Lambda_{t})};
\end{align*}
let us only note that we need the continuity of $\phi(\cdot)$ at the point $t=0$ because, in the final steps of computation, we get a term of form $\phi(c_{B}\sum_{j=0}^{l}|e^{i \langle \lambda_{j},\tau\rangle }-1|),$ which tends to zero as $\tau \rightarrow 0+.$
The proof of part (a) in  (iii) follows from a relatively simple argumentation involving the decomposition used for proving (ii), the given assumptions and the fact that the trigonometric polynomial $P_{k}(\cdot;\cdot)$ is Bohr ${\mathcal B}$-almost periodic due to \cite[Proposition 6.1.25(iv)]{nova-selected}; the proof of part (b) in (iii) is quite similar because the prescribed assumptions imply that the trigonometric polynomial $P_{k}(\cdot;\cdot)$ is  Bohr $({\mathcal B},\Lambda \cap \Delta_{n})$-almost periodic due to \cite[Proposition 6.1.25(v)]{nova-selected} (cf. also \cite[Definition 6.1.9, Definition 6.1.14]{nova-selected} for the notion).
\end{proof}

Suppose that ${\mathcal B}$ consists of bounded subsets of $X$, $\Lambda$ is unbounded, $F,\ G\in e-({\mathcal B},\phi,{\rm F})-B^{p(\cdot)}(\Lambda \times X : Y)$, for every fixed element $x\in X,$ 
the function $\phi(\cdot)$ is monotonically increasing, $\phi(x+y)\leq \phi(x)+\phi(y)$ for all $x,\ y\geq 0,$ $\limsup_{t\rightarrow +\infty}{\rm F}(t)=0,$ and
there exists a finite real constant $t_{0}>0$ such that ${\rm F}(t)\leq [\|1\|_{L^{p(\cdot)}(\Lambda_{t})}]^{-1},$ $t\geq t_{0}.$ Due to Proposition \ref{anat}(i), we have that  
the function $F(\cdot;\cdot)$ is Besicovitch-$(p,\phi,{\rm F},{\mathcal B})$-bounded. Let a set $B\in {\mathcal B}$ be fixed. Then 
$$
d_{B}(F,G):=\limsup_{t\rightarrow +\infty}{\rm F}(t)\sup_{x\in B}\Bigl[ \phi\Bigl( \bigl\| F({\cdot};x)-G({\cdot};x) \bigr\|_{Y} \Bigr)\Bigr]_{L^{p(\cdot)}(\Lambda_{t})}
$$
defines a pseudometric on the set $
e-({\mathcal B},\phi,{\rm F})-B^{p(\cdot)}(\Lambda \times X : Y).$ Using the idea from the original proof of J. Marcinkiewicz \cite{marcin}, we can prove the following theorem (see also \cite[pp. 249--252]{188}):

\begin{thm}\label{gs}
Let the requirements stated in the previous paragraph hold. Then, for every set $B\in {\mathcal B}$, the pseudometric space $(e-({\mathcal B},\phi,{\rm F})-B^{p(\cdot)}(\Lambda \times X : Y),d_{B})$ is complete.
\end{thm}

\noindent {\bf The classes with $\phi(x)\equiv x^{\alpha}$, $\alpha>0.$}
Without any doubt, the most important case in Definition \ref{ebe} and Definition \ref{prespansko} is that one in which we have $\phi(x)\equiv x^{\alpha}$ for some real number $\alpha>0.$ If so, then all requirements necessary for applying Proposition \ref{anat} and the statements stated preceding it hold. The assumptions of Theorem \ref{gs} hold in case $\alpha \in (0,1],$ when we can provide some proper generalizations of the usual notion of Besicovitch-$P$-almost periodicity. For example, if $1\leq P<+\infty,$ $1\leq p<+\infty,$ $\alpha p\in (0,1)$ and the function $F: {\mathbb R} \rightarrow Y$ is Besicovitch-$P$-almost periodic, then the H\"older inequality implies that, for every trigonometric polynomial $P(\cdot)$, we have:
\begin{align*}
\Biggl( \frac{1}{2t}\int^{t}_{-t}\bigl\| F(t)-P(t)\bigr\|_{Y}^{\alpha p}\, dt \Biggr)^{1/p}\leq \Biggl( \frac{1}{2t}\int^{t}_{-t}\bigl\| F(t)-P(t)\bigr\|_{Y}^{P}\, dt \Biggr)^{\alpha/P},\quad t>0, 
\end{align*}
 so that $F\in e-(x^{\alpha},t^{-(1/p)})-B^{p}({\mathbb R} : Y).$ The converse statement does not hold in general, as the following illustrative example shows:

\begin{example}\label{maric0} (cf. also Example \ref{maric} below)
Let $\zeta>1/2$ and $\alpha \zeta \in (0,1/2).$
Define the function $F : {\mathbb R} \rightarrow {\mathbb R}$ by $F(t):=m^{\zeta}$ if $t\in [m^{2},m^{2}+\sqrt{|m|})$ for some $m\in {\mathbb Z},$ and $F(t):=0,$ otherwise. Then it can be simply shown that the function $F(\cdot)$ is not Besicovitch bounded and therefore not Besicovitch almost periodic. On the other hand, we have $F\in 
PAP_{0,p}({\mathbb R},t^{-1},x^{\alpha})
\subseteq
 e-(x^{\alpha},t^{-1})-B^{1}({\mathbb R} : {\mathbb C}).$
\end{example}

Let the numbers $\alpha>0$ and $ \beta>0$ be arbitrary. Using the functions $\phi(x)\equiv x^{\alpha/p}$ and ${\rm F}(t)\equiv t^{-\beta/p}$ in our approach, we can consider the generalized Besicovitch class $B_{\alpha,\beta}({\mathbb R} : Y)$ consisting of those Lebesgue measurable functions $F: {\mathbb R} \rightarrow Y$ such that, for every $\epsilon>0,$ there exist a trigonometric polynomial $P(\cdot)$ and a real number $t_{0}>0$ such that
$$
\int_{-t}^{t}\bigl\| F(s)-P(s) \bigr\|_{Y}^{\alpha}\, ds \leq \epsilon t^{\beta},\quad t\geq t_{0};
$$
a multi-dimensional generalization can be introduced analogously. Fairly complete analysis of the generalized Besicovitch class $B_{\alpha,\beta}({\mathbb R} : Y)$ and its multi-dimensional analogues is without scope of this paper (let us only observe here that the space $W_{\alpha}$, considered by M. A. Picardello  \cite{pic} in the usual one-dimensional setting with $0<\alpha \leq 1,$ is nothing else but the space $B_{1,\alpha}({\mathbb R} : {\mathbb C})$). 

We will provide the main details of the proof of the following proposition for the sake of completeness:

\begin{prop}\label{sdffds}
Suppose that $p,\ q,\ r \in [1,\infty),$ $1/r=1/p+1/q$, ${\rm F}_{1}(t)\equiv t^{-n/p}$,
${\rm F}_{2}(t)\equiv t^{-n/q},$
${\rm F}(t)\equiv t^{-n/r}$, $\phi(x)\equiv x^{\alpha}$ for some real number $\alpha>0,$ and any set $B$ of collection ${\mathcal B}$ is bounded in $X.$ If $F_{1}\in e-({\mathcal B},\phi,{\rm F}_{1})-B^{p}(\Lambda \times X : {\mathbb C})$ and $F_{2}\in e-({\mathcal B},\phi,{\rm F}_{2})-B^{q}(\Lambda \times X : Y),$ then the function $F : \Lambda \times X \rightarrow Y,$ given by $F({\bf t};x):=F_{1}({\bf t};x)F_{2}({\bf t};x),$ ${\bf t}\in \Lambda,$ $x\in X,$ belongs to the class $e-({\mathcal B},\phi,{\rm F})-B^{r}(\Lambda \times X : Y).$
\end{prop}

\begin{proof}
Let $\epsilon>0$ and $B\in {\mathcal B}$ be given. Then there exist a finite real number $t_{0}>0,$ a scalar-valued trigonometric polynomial $P_{1}(\cdot;\cdot)$ and a $Y$-valued trigonometric polynomial $P_{2}(\cdot;\cdot)$ such that, for every $x\in B,$ we have
\begin{align}\label{zoki}
\Bigl[\phi\bigl(\bigl\| F_{1}(\cdot;x)-P_{1}(\cdot;x)\bigr\|_{Y}\bigr)\Bigr]_{L^{p}(\Lambda_{t})}\leq \epsilon t^{n/p},\quad t\geq t_{0},
\end{align}
and
\begin{align}\label{zoki1}
\Bigl[\phi\bigl(\bigl\| F_{2}(\cdot;x)-P_{2}(\cdot;x)\bigr\|_{Y}\bigr)\Bigr]_{L^{q}(\Lambda_{t})}\leq \epsilon t^{n/q},\quad t\geq t_{0}.
\end{align}
Clearly, $P_{1}(\cdot;\cdot)P_{2}(\cdot;\cdot)$ is a $Y$-valued trigonometric polynomial.
Applying Proposition \ref{anat}(i), we get that the function $P_{1}(\cdot;\cdot)$ is Besicovitch-$(p,\phi,{\rm F}_{1},{\mathcal B})$-bounded and the function $F_{2}(\cdot;\cdot)$ is Besicovitch-$(q,\phi,{\rm F}_{2},{\mathcal B})$-bounded. Keeping in mind that $1/r=1/p+1/q$, the final conclusion simply follows using this fact, \eqref{zoki}-\eqref{zoki1}, the existence of a finite real number $c_{\alpha}>0$ such that
\begin{align*}
\phi\Bigl(\bigl\| & F_{1}(\cdot;x)F_{2}(\cdot;x)-P_{1}(\cdot;x)P_{2}(\cdot;x)\bigr\|_{Y}\Bigr)
\\& \leq c_{\alpha}\Biggl[\phi\Bigl(\bigl| F_{1}(\cdot;x)-P_{1}(\cdot;x) \bigr|\Bigr) \cdot \phi\Bigl(\bigl\|F_{2}(\cdot;x)\bigr\|_{Y}\Bigr)
\\&+\phi\Bigl(\bigl| P_{1}(\cdot;x) \bigr|\Bigr) \cdot \phi\Bigl(\bigl\| F_{2}(\cdot;x)-P_{2}(\cdot;x) \bigr\|_{Y}\Bigr)\Biggr],
\end{align*}
and the H\"older inequality.
\end{proof}

We continue by providing the following illustrative application of Proposition \ref{sdffds}:

\begin{example}\label{maric}
Suppose that $1\leq p_{1},...,p_{n},p<+\infty$ and $1/p=1/p_{1}+1/p_{2}+...+1/p_{n}.$ Define the function $F_{j} : {\mathbb R} \rightarrow {\mathbb R}$ by $F_{j}(t):=m^{1/2p_{j}}$ if $t\in [m^{2},m^{2}+\sqrt{|m|})$ for some $m\in {\mathbb Z},$ and $F_{j}(t):=0,$ otherwise ($1\leq j\leq n$). Then we know that the function $F_{j}(\cdot)$ is Besicovitch-$p_{j}$-almost periodic but not Besicovitch-$q$-almost periodic if $q>p_{j};$ see \cite[p. 42]{berta} and \cite[Example 6.24]{deda}. Define $F({\bf t}):=F_{1}(t_{1})\cdot F_{2}(t_{2})\cdot ... \cdot F_{n}(t_{n}),$ ${\bf t}\in {\mathbb R}^{n}.$ Applying Proposition \ref{sdffds} and a simple argumentation, it follows that the function $F(\cdot)$ is Besicovitch-$p$-almost periodic but not Besicovitch-$q$-almost periodic if $q>p.$
\end{example}

Sometimes we need the value of coefficient $p=+\infty$ in Proposition \ref{sdffds} and sometimes the usual choice ${\rm F}(t)\equiv t^{-n/p}$ is wrong if the region $\Lambda$ is bounded in direction of some real axes:

\begin{example}\label{maricm}
Suppose that $1\leq p<+\infty,$ the function $f : [0,2\pi]\rightarrow {\mathbb R}$ is absolutely continuous and the function $g : [0,\infty)\rightarrow Y$ is Besicovitch-$p$-almost periodic. Since the Fourier series of function $f(\cdot)$ converges uniformly to this function, arguing as in the proof of Proposition \ref{sdffds} we may conclude that the function $F(x,y):=f(x)g(y),$ $(x,y)\in \Lambda \equiv [0,2\pi]\times [0,\infty) \rightarrow Y$ belongs to the class $e-(x,t^{-1/p})-B^{p}(\Lambda : Y).$
\end{example}

Further on, the composition principles for one-dimensional Besicovitch-$p$-almost periodic functions have been analyzed for the first time by M. Ayachi and J. Blot in \cite[Lemma 4.1]{ayachi}. In the following theorem, we consider the Besicovitch-$p$-almost periodicity of the multi-dimensional Nemytskii operator $W : {\mathbb R}^{n} \times X \rightarrow Z,$ given by
\begin{align}\label{defin}
W({\bf t};x):=G({\bf t}; F({\bf t};x)),\quad {\bf t}\in {\mathbb R}^{n},\ x\in X,
\end{align}
where $F : {\mathbb R}^{n} \times X \rightarrow Y$ and $G : {\mathbb R}^{n} \times Y \rightarrow Z.$ We follow the ideas from \cite{ayachi} in (i):

\begin{thm}\label{krompo}
Suppose that $1\leq p,\ q<+\infty,$ $\alpha>0,$ $p=\alpha q,$ ${\rm F}(t)\equiv t^{-n/p},$ $\phi(x)\equiv x^{\zeta}$ for some real number $\zeta>0,$ $F\in e-({\mathcal B},\phi,{\rm F})-B^{p}({\mathbb R}^{n} \times X : Y),$ and ${\mathcal B}$ is the collection consisting of all bounded subsets of $X.$ 
\begin{itemize}
\item[(i)]
Suppose that $G : {\mathbb R}^{n} \times Y \rightarrow Z$ is Bohr ${\mathcal B}$-almost periodic and there exists a finite real constant $a>0$ such that 
\begin{align}\label{pravilo}
\Bigl\| G({\bf t}; y)-G({\bf t}; y')\Bigr\|_{Z}
\leq a \bigl\| y-y'\bigr\|_{Y}^{\alpha},\quad {\bf t}\in {\mathbb R}^{n},\ y,\ y'\in Y.
\end{align}
Then the function $W(\cdot;\cdot),$ given by \eqref{defin}, belongs to the class $e-({\mathcal B},\phi,t^{-n/q})-B^{q}({\mathbb R}^{n} \times X : Z).$
\item[(ii)] Define
$$
{\mathcal B}':=\Biggl\{ \bigcup_{{\bf t}\in {\mathbb R}^{n}}F({\bf t};B) \, ; \, B\in {\mathcal B}\Biggr\}.
$$
By $e-({\mathcal B}',\phi,t^{-n/q})-B^{q}_{a,\alpha}({\mathbb R}^{n} \times Y : Z)$ we denote the class of all functions $G_{1}(\cdot;\cdot)$ such that for each set $B'\equiv \cup_{{\bf t}\in {\mathbb R}^{n}}F({\bf t};B)\in {\mathcal B}'$ there exists a sequence of Bohr ${\mathcal B}$-almost periodic functions $(G_{1}^{k}(\cdot;\cdot))$ such that \eqref{pravilo} holds with the function $G(\cdot;\cdot)$ replaced therein by the function $G_{1}^{k}(\cdot;\cdot)$ for all $k\in {\mathbb N},$ the equation \eqref{jew} holds 
with the function $F(\cdot;\cdot)$ replaced therein by the function
$G(\cdot;\cdot)$, the polynomial $P_{k}(\cdot;\cdot)$ replaced therein by the function $G_{1}^{k}(\cdot;\cdot),$ the set $B$ replaced therein with the set $B',$ and the exponent $p(\cdot)$ replaced therein by the constant exponent $q$. If $G\in e-({\mathcal B}',\phi,t^{-n/q})-B^{q}_{a,\alpha}({\mathbb R}^{n} \times Y : Z),$ then
$W\in e-({\mathcal B},\phi,t^{-n/q})-B^{q}({\mathbb R}^{n} \times X : Z).$
\end{itemize}
\end{thm}

\begin{proof}
Let $\epsilon>0$ and $B\in {\mathcal B}$ be given. Then there exist a trigonometric polynomial $P(\cdot;\cdot)$ and a finite real number $t_{0}>0$ such that 
$$
\sup_{x\in  B}\int_{[-t,t]^{n}}\bigl\| F({\bf t};x)-P({\bf t};x)\bigr\|_{Y}^{\zeta p}\, d{\bf t}<\epsilon t^{n},\quad t\geq t_{0}.
$$
We will first prove (i).
Since we have assumed that \eqref{pravilo} holds and ${\mathcal B}$ is the collection consisting of all bounded subsets of $X$, the argumentation contained in the proof of \cite[Theorem 6.1.47, Corollary 6.1.48]{nova-selected} shows that the function $W_{1} : {\mathbb R}^{n} \times X \rightarrow Z,$ given by $W_{1}({\bf t};x):=G({\bf t};P({\bf t};x)),$ ${\bf t}\in {\mathbb R}^{n},$ $x\in X,$ is Bohr ${\mathcal B}$-almost periodic. Then the final conclusion simply follows by observing that $p=\alpha q,$ using the next estimate which holds for any $t>0$ and $x\in B;$ see \eqref{pravilo}:
\begin{align*}
\int_{[-t,t]^{n}}\bigl\| W({\bf t};x)-W_{1}({\bf t};x)\bigr\|_{Z}^{\zeta q}\, d{\bf t}\leq a^{\zeta q}\int_{[-t,t]^{n}}\bigl\| F({\bf t};x)-P({\bf t};x)\bigr\|_{Y}^{\zeta p}\, d{\bf t}.
\end{align*}
In order to prove (ii), it suffices to apply (i) and use the decomposition{\small
\begin{align*}
&\phi\Bigl(\bigl\|  G({\bf t};F({\bf t};x))-G_{k}({\bf t};P({\bf t};x))\bigr\|_{Z}\Bigr)
\\& \leq 2^{\zeta}\Biggl[\phi\Bigl(\bigl\| G({\bf t};F({\bf t};x))-G_{k}({\bf t};F({\bf t};x))\bigr\|_{Z}\Bigr)+\phi\Bigl(\bigl\|  G_{k}({\bf t};F({\bf t};x))-G_{k}({\bf t};P({\bf t};x))\bigr\|_{Z}\Bigr)\Biggr],
\end{align*}}
where $G_{k}(\cdot;\cdot)$ properly approximates $G(\cdot;\cdot)$
in the space $e-({\mathcal B}',\phi,t^{-n/q})-B^{q}({\mathbb R}^{n} \times Y : Z).$
\end{proof}

\subsection{Multi-dimensional Besicovitch normal type functions}\label{normalno}

The notion of a Besicovitch $p$-normal function $f : {\mathbb R} \rightarrow {\mathbb C}$ was introduced by R. Doss in \cite{doss0} and later reconsidered by the same author in \cite{doss1}; cf. also
\cite[Subsection 8.3.2, Definition 8.3.18]{nova-selected}, where we have recently analyzed the concept Weyl $p$-almost automorphy (of type $2$) without limit functions. In this subsection, we will consider the following notion:

\begin{defn}\label{dfggg}
Suppose that ${\mathrm R}$ is any collection of sequences in $\Lambda'',$ $F: \Lambda \times X \rightarrow Y,$ $\phi : [0,\infty) \rightarrow [0,\infty)$ and ${\rm F} : (0,\infty) \rightarrow (0,\infty).$ Then we say that the function $F(\cdot;\cdot)$ is
Besicovitch$-({\mathrm R}, {\mathcal B},\phi,{\rm F})-B^{p(\cdot)}$-normal
if and only if for every set $B\in {\mathcal B}$ and for every sequence $({\bf b}_{k})_{k\in {\mathbb N}}$ in ${\mathrm R}$ there exists a subsequence $({\bf b}_{k_{m}})_{m\in {\mathbb N}}$ of $({\bf b}_{k})_{k\in {\mathbb N}}$ such that, for every $\epsilon>0,$ there exists an integer $m_{0}\in {\mathbb N}$ such that, for every integers $m,\ m'\geq m_{0},$ we have 
\begin{align*}
\limsup_{t\rightarrow +\infty}{\mathrm F}(t)\sup_{x\in B}\Biggl[ \phi\Bigl( \bigl\| F({\bf t}+{\bf b}_{k_{m}};x)-F({\bf t}+{\bf b}_{k_{m'}};x)\bigr\|_{Y} \Bigr)\Biggr]_{L^{p({\bf t})}(\Lambda_{t})} <\epsilon.
\end{align*}
\end{defn}
The usual notion of Besicovitch-$p$-normality for the function $F : \Lambda \rightarrow Y,$ where $1\leq p<+\infty,$ is obtained by plugging $\phi(x)\equiv x$ and ${\rm F}(t)\equiv t^{-n/p},$ with 
${\mathrm R}$ being the collection of all sequences in $\Lambda''.$

In the sequel, we will occasionally use the following conditions:

\begin{itemize}
\item[(I)] $\phi(\cdot)$ is monotonically increasing, continuous at the point $t=0,$ and $p(\cdot)\equiv p\in [1,\infty).$
\item[(II)]  There exists $c\in (0,1]$ such that $\phi(x+y)\leq c[\phi(x)+\phi(y)]$ for all $x,\ y\geq 0$, and there exists a function $\varphi : [0,\infty) \rightarrow [0,\infty)$ such that $\phi(xy)\leq \phi(x)\varphi(y)$ for all $x,\ y\geq 0$ and $D:=\sup_{m\in {\mathbb N}}[m\varphi(1/m)]<+\infty.$
\item[(III)] $\limsup_{t\rightarrow +\infty}[{\rm F}(t)m(\Lambda_{t})^{1/p}]<+\infty$ and, for every real number $a>0,$ we have $\limsup_{t\rightarrow +\infty}[{\rm F}(t)/{\rm F}(t+a)]\leq 1.$
\end{itemize}

It is clear that (II) holds provided that $\phi(x)\equiv x^{\alpha}$ for some real number $\alpha\geq 1$ as well as that (II) does not hold if $\phi(x)\equiv x^{\alpha}$ for some real number $\alpha\in (0,1).$

Repeating verbatim the argumentation contained in the proof of \cite[Theorem 6.3.19]{nova-selected}, where we have analyzed the concept Weyl $({\mathrm R},{\mathcal B},p)$-normality, the following result can be deduced without any substantial difficulties: 

\begin{prop}\label{raj-sed}
Suppose that  $F: \Lambda \times X \rightarrow Y$, $F\in e-({\mathcal B},\phi,{\rm F})-B^{p(\cdot)}(\Lambda \times X : Y),$ and conditions \emph{(I)-(III)} hold. Then $F(\cdot;\cdot)$ is Besicovitch$-({\mathrm R}, {\mathcal B},\phi,{\rm F})-B^{p(\cdot)}$-normal.
\end{prop}

Even in the usual one-dimensional framework, we know that the converse statement of Proposition \ref{raj-sed} is not true in general (see, e.g., \cite{deda}) as well as that the usual Besicovitch-$p$-normality does not imply Besicovitch-$p$-continuity (see, e.g., \cite{doss1}). Further on, let $k\in {\mathbb N}$ and $F_{i} : \Lambda \times X \rightarrow Y_{i}$ ($1\leq i\leq k$). Then we define the function $(F_{1},..., F_{k}) : \Lambda \times X \rightarrow Y_{1}\times ... \times Y_{k}$ by
$$
\bigl(F_{1},..., F_{k}\bigr)({\bf t} ;x):=\bigl(F_{1}({\bf t};x) , ..., F_{k}({\bf t};x) \bigr),\quad {\bf t} \in \Lambda,\ x\in X.
$$

The following result is trivial and its proof is therefore omitted:

\begin{prop}\label{kursk-kursk}
Suppose that $k\in {\mathbb N}$, $\emptyset \neq \Lambda \subseteq {\mathbb R}^{n},$ and we have that, for any sequence which belongs to 
${\mathrm R},$ any its subsequence also belongs to ${\mathrm R}.$
If the function $F_{i}(\cdot;\cdot)$ is Besicovitch$-({\mathrm R}, {\mathcal B},\phi,{\rm F})-B^{p(\cdot)}$-normal for $1\leq i\leq k$, then the function $(F_{1},..., F_{k})(\cdot;\cdot)$ is also Besicovitch$-({\mathrm R}, {\mathcal B},\phi,{\rm F})-B^{p(\cdot)}$-normal.
\end{prop}

The interested reader may try to formulate certain conditions ensuring that the limit function of a sequence of uniformly convergent Besicovitch$-({\mathrm R}, {\mathcal B},\phi,{\rm F})-B^{p(\cdot)}$-normal functions (the functions belonging to the class $e-({\mathcal B},\phi,{\rm F})-B^{p(\cdot)}(\Lambda \times X : Y)$) is also Besicovitch$-({\mathrm R}, {\mathcal B},\phi,{\rm F})-B^{p(\cdot)}$-normal (belongs to the class $e-({\mathcal B},\phi,{\rm F})-B^{p(\cdot)}(\Lambda \times X : Y)$); see  \cite{nova-selected} for many results of this type. 

Several structural properties of functions belonging to the class $e-({\mathcal B},\phi,{\rm F})-B^{p(\cdot)}(\Lambda \times X : Y)$ can be simply reformulated for the class of Besicovitch$-({\mathrm R}, {\mathcal B},\phi,{\rm F})-B^{p(\cdot)}$-normal functions. For example, we have the following analogue of Proposition \ref{sdffds}:

\begin{prop}\label{sdffds1}
Suppose that $p,\ q,\ r \in [1,\infty),$ $1/r=1/p+1/q$, ${\rm F}_{1}(t)\equiv t^{-n/p}$,
${\rm F}_{2}(t)\equiv t^{-n/q},$
${\rm F}(t)\equiv t^{-n/r}$, $\phi(x)\equiv x^{\alpha}$ for some real number $\alpha>0$ and, for any sequence which belongs to 
${\mathrm R},$ any its subsequence also belongs to ${\mathrm R}.$
If the function $F_{1} : \Lambda \times X\rightarrow {\mathbb C}$ is
Besicovitch$-({\mathrm R}, {\mathcal B},\phi,{\rm F}_{1})-B^{p}$-normal and  Besicovitch-$(p,\phi,{\rm F}_{1}, {\mathcal B})$-bounded as well as the function 
$F_{2} : \Lambda \times X\rightarrow Y$ is
Besicovitch$-({\mathrm R}, {\mathcal B},\phi,{\rm F}_{2})-B^{q}$-normal and  Besicovitch-$(q,\phi,{\rm F}_{2}, {\mathcal B})$-bounded,
then the function $F : \Lambda \times X \rightarrow Y,$ given by $F({\bf t};x):=F_{1}({\bf t};x)F_{2}({\bf t};x),$ ${\bf t}\in \Lambda,$ $x\in X,$ is Besicovitch$-({\mathrm R}, {\mathcal B},\phi,{\rm F}_{1})-B^{r}$-normal.
\end{prop}

\begin{proof}
Let a set $B\in {\mathcal B}$ and a sequence $({\bf b}_{k})_{k\in {\mathbb N}}$ in ${\mathrm R}$ be given.
Keeping in mind the proof of Proposition \ref{sdffds} and our assumption that for any sequence which belongs to 
${\mathrm R}$ any its subsequence also belongs to ${\mathrm R}$, it suffices to show that there exist two sufficiently large real numbers $t_{0}>0$ and $M>0$ such that 
\begin{align}\label{qwerqwer}
\sup_{x\in B,k\in {\mathbb N}}\Bigl[\phi\bigl(|F_{1}(\cdot+{\bf b}_{k};x)|\bigr)\Bigr]_{L^{p}(\Lambda_{t})} \leq Mt^{n/p},\quad t\geq t_{0}
\end{align}
and 
\begin{align}\label{qwerqwer1}
\sup_{x\in B,k\in {\mathbb N}}\Bigl[\phi\bigl(\|F_{2}(\cdot+{\bf b}_{k};x)\|_{Y}\bigr)\Bigr]_{L^{q}(\Lambda_{t})}\leq Mt^{n/q},\quad t\geq t_{0}.
\end{align}
Since the function $F_{1} : \Lambda \times X\rightarrow {\mathbb C}$ is
Besicovitch-$(p,\phi,{\rm F}_{1}, {\mathcal B})$-bounded, it can be simply proved that 
\begin{align*}
\sup_{x\in B}\Bigl[\phi\bigl(|F_{1}(\cdot+{\bf b}_{k};x)|\bigr)\Bigr]_{L^{p}(\Lambda_{t})}\ \leq Mt^{n/p},\quad t\geq t_{0}.
\end{align*}
Moreover, we have the existence of a finite real number $t_{0}>0$ and an integer $k_{0}\in {\mathbb N}$ such that, for every integers $k,\ k'\geq k_{0},$ we have 
\begin{align*}
\sup_{x\in B}\Bigl[ \phi\Bigl( \bigl| F_{1}(\cdot+{\bf b}_{k};x)-F_{1}(\cdot+{\bf b}_{k'};x)\bigr| \Bigr)\Bigr]_{L^{p}(\Lambda_{t})} < Mt^{n/p},\quad t\geq t_{0}.
\end{align*}
Further on, there exists a finite real constant $c_{\alpha}>0$ such that, for every integers $k,\ k'\geq k_{0},$ we have:
\begin{align*}
\Bigl[ &\phi\bigl(|F_{1}(\cdot+{\bf b}_{k};x)|\bigr)-\phi\bigl(|F_{1}(\cdot+{\bf b}_{k'};x)|\bigr)\Bigr]_{L^{p}(\Lambda_{t})}
\\& \leq c_{\alpha}\Bigl[ \phi\Bigl(\Bigr| |F_{1}(\cdot+{\bf b}_{k};x)|-|F_{1}(\cdot+{\bf b}_{k'};x)|\Bigr|\Bigr)\Bigr]_{L^{p}(\Lambda_{t})}
\\ & \leq c_{\alpha} \Bigl[ \phi\Bigl(\bigl|F_{1}(\cdot+{\bf b}_{k};x)-F_{1}(\cdot+{\bf b}_{k'};x)\bigr|\Bigr)\Bigr]_{L^{p}(\Lambda_{t})}\leq c_{\alpha}Mt^{n/p},\quad t\geq t_{0},\ x\in B.
\end{align*}
This simply implies \eqref{qwerqwer} because
\begin{align*}
\sup_{x\in B}&\Bigl[\phi\bigl(|F_{1}(\cdot+{\bf b}_{k};x)|\bigr)\Bigr]_{L^{p}(\Lambda_{t})}
\\& \leq \sup_{x\in B}\Bigl[\phi\bigl(|F_{1}(\cdot+{\bf b}_{k_{0}};x)|\bigr)\Bigr]_{L^{p}(\Lambda_{t})}+c_{\alpha}Mt^{n/p},\quad t\geq t_{0},\ k\geq k_{0}.
\end{align*}
The estimate \eqref{qwerqwer1} can be proved analogously, finishing the proof.
\end{proof}

Concerning the asumptions on Besicovitch boundedness used in the formulation of Proposition \ref{sdffds1}, we would like to recall that A. Haraux and P. Souplet have proved (see \cite[Theorem 1.1]{haraux}) that the
function $f: {\mathbb R}\rightarrow {\mathbb R},$ given by
\begin{align}\label{voja}
f(t):=\sum_{n=1}^{\infty}\frac{1}{n}\sin^{2}\Bigl(\frac{t}{2^{n}} \Bigr)\, dt,\quad t\in {\mathbb R},
\end{align}
is uniformly continuous, uniformly recurrent (the sequence $(\alpha_{k}\equiv 2^{k}\pi)_{k\in {\mathbb N}}$) can be chosen in definition of uniform recurrence) and Besicovitch  unbounded; see \cite{nova-selected} for the notion. Let ${\rm R}$ denote the collection consisting of the seqeunce $(\alpha_{k})_{k\in {\mathbb N}}$ and all its subsequences. Then the function $f(\cdot)$ is 
Besicovitch$-({\mathrm R},x,t^{-1/p})-B^{p}$-normal but not  Besicovitch-$(p,x,t^{1/p})$-bounded for any finite exponent $p\geq 1.$ Furthermore, we have the following:

\begin{example}\label{weyl-auto}
Suppose that $p\in [1,\infty),$ $\sigma \in (0,1),$ $F(x):=|x|^{\sigma},$ $x\in {\mathbb R},$ and $a>1-(1-\sigma)p>0.$ Then we know that the function $F(\cdot)$ is not Besicovitch-$p$-bounded and that, for every $t\in {\mathbb R}$ and $\omega\in {\mathbb R},$ we have:
\begin{align}\label{838}
\lim_{l\rightarrow +\infty}l^{-a}\int^{l}_{-l}\Bigl| \bigl| x+t+\omega \bigr|^{\sigma}-\bigl| x+t \bigr|^{\sigma} \Bigr|^{p}\, dx=0;
\end{align}
see \cite[Theorem 8.3.8]{nova-selected} and its proof. Let ${\mathrm R}$ denote the collection of all sequences in ${\mathbb R}$ and let ${\rm F}(t)\equiv t^{-a/p}.$ Then the limit equality \eqref{838} simply implies that the function $F(\cdot)$ is Besicovitch$-({\mathrm R}, x,{\rm F})-B^{p}$-normal. Hence, the usual Besicovitch-$p$-normality of a function $F(\cdot)$ does not imply its Besicovitch-$p$-boundedness as well.
\end{example}

Now we would like to recall that any Doss-$p$-almost periodic function $F : [0,\infty) \rightarrow Y$, where $p\in [1,\infty),$ can be extended to a Doss-$p$-almost periodic function $\tilde{F} : {\mathbb R} \rightarrow Y$ defined by $\tilde{F}(t):=0,$ $t<0$ (cf. \cite{doss-rn} for the notion used in this paragraph).
A similar type of extension can be achieved in a much more general situation; for example, we know that, under certain reasonable conditions, any
Doss-$(p,\phi,{\rm F},{\mathcal B},\Lambda',\rho)$-almost periodic function $F : \Lambda \times X \rightarrow Y$ can be extended to a 
Doss-$(p,\phi,{\rm F},{\mathcal B},\Lambda',\rho_{1})$-almost periodic function $\tilde{F} : {\mathbb R}^{n} \times X \rightarrow Y,$ defined by $\tilde{F}({\bf t}):=0,$ $t\notin \Lambda,$ $\tilde{F}({\bf t}):=F({\bf t}),$ $t\in \Lambda,$
with $\rho_{1}:=\rho \cup \{(0,0)\}$ (the corresponding analysis from \cite{doss-rn} contains small typographical errors that will be corrected in our forthcoming monograph \cite{advances}). 

We would like to emphasize that a similar analysis cannot be carried out for Besicovitch almost periodic type functions. We close this section by introducing the following notion and raising the following issue:

\begin{defn}\label{bes-ext}
Let $\emptyset \neq \Lambda \subseteq {\mathbb R}^{n}$, and let ${\mathcal C}_{\Lambda}=e-({\mathcal B},\phi,{\rm F})-B^{p(\cdot)}(\Lambda  : Y)$ or ${\mathcal C}_{\Lambda}$ be the class consisting of all Besicovitch$-({\mathrm R}, {\mathcal B},\phi,{\rm F})-B^{p(\cdot)}$-normal functions.
Then we say that the set $\Lambda$ is admissible with respect to the class ${\mathcal C}_{\Lambda}$ if and only if for any complex Banach space $Y$ and for any function $F : \Lambda \rightarrow Y$ 
there exists a function $\tilde{F}\in {\mathcal C}_{{\mathbb R}^{n}}$ such that $\tilde{F}({\bf t})=F({\bf t})$ for all ${\bf t}\in \Lambda.$ 
\end{defn}\index{set!admissible with respect to the class ${\mathcal C}_{\Lambda}$}

\noindent {\bf Problem.} It is still not known whether the set $[0,\infty)\subseteq {\mathbb R}$ is admissible with respect to the class of Besicovitch-$p$-almost periodic functions, i.e., whether a Besicovitch-$p$-almost periodic function $f : [0,\infty) \rightarrow Y$ can be extended to a Besicovitch-$p$-almost periodic function $\tilde{f} : {\mathbb R} \rightarrow Y$ defined on the whole real line ($1\leq p<\infty$). We would like to ask here a more general question: Is it true that a convex polyhedral
$
\Lambda
$ in ${\mathbb R}^{n}$
is admissible with respect to the class of multi-dimensional Besicovitch-$p$-almost periodic functions (Besicovitch-$p$-normal functions)?\vspace{0.1cm}

\section{Besicovitch-Doss almost periodicity}\label{novs}

In this section, we discuss and reexamine several structural results established by R. Doss in \cite{doss}-\cite{doss1}. We work in the multi-dimensional setting here, considering especially the following conditions: 

\begin{itemize}  
\item[(A)] For every $B\in {\mathcal B}$ and $a\in \Lambda'',$ there exists a function $F^{(a)}_{B} : \Lambda \times X \rightarrow Y$ such that $F^{(a)}_{B}(\cdot;x)$ is $a$-periodic for every fixed element $x\in B$, i.e., $F^{(a)}_{B}({\bf t}+a;x)=F^{(a)}_{B}({\bf t};x)$ for all ${\bf t}\in \Lambda,$ $x\in B,$ $\|F^{(a)}_{B}({\bf t};x)\|_{Y}\in L^{p({\bf t})}(\Lambda_{t})$ for all $t>0,$ $x\in B$, and
\begin{align}\label{maleni}
\lim_{k\rightarrow +\infty}\limsup_{t\rightarrow +\infty}{\rm F}(t)\sup_{x\in B}\Biggl[\phi\Biggl(\Biggr\| \frac{1}{k}\sum_{j=0}^{k-1}F({\bf t}+ja;x)-F^{(a)}_{B}({\bf t};x)\Biggr\|_{Y}\Biggr)\Biggr]_{L^{p({\bf t})}(\Lambda_{t})}=0.
\end{align}
\item[(A)$_{\infty}$] For every $B\in {\mathcal B}$ and $a\in \Lambda'',$ there exists a function $F^{(a)}_{B} : \Lambda \times X \rightarrow Y$ such that $F^{(a)}_{B}(\cdot;x)$ is $a$-periodic for every fixed element $x\in B$, $\|F^{(a)}_{B}(\cdot;x)\|_{Y}\in L^{\infty}({\mathbb R}^{n})$ for all $x\in B$, and \eqref{maleni} holds.
\item[(AS)] For every $B\in {\mathcal B}$ and $a=(a_{1},a_{2},...,a_{n})\in \Lambda''$ such that 
$a_{j}e_{j}\in \Lambda''$ for all $j\in {\mathbb N}_{n},$
there exists a function $F^{(a)}_{B} : \Lambda \times X \rightarrow Y$ such that $F^{(a)}_{B}(\cdot;x)$ is $(a_{j})_{j\in {\mathbb N}_{n}}$-periodic for every fixed element $x\in B$, i.e., $F^{(a)}_{B}({\bf t}+a_{j}e_{j};x)=F^{(a)}_{B}({\bf t};x)$ for all ${\bf t}\in \Lambda,$ $x\in B,$ $j\in {\mathbb N}_{n},$ $\|F^{(a)}_{B}({\bf t};x)\|_{Y}\in L^{p({\bf t})}(\Lambda_{t})$ for all $t>0,$ $x\in B$, and \eqref{maleni} holds.
\end{itemize}

It is clear that (AS) implies (A) as well as that both conditions are equivalent in the one-dimensional setting; it is also clear that (A)$_{\infty}$ implies (A). Further on,
for every Lebesgue measurable set $E\subseteq {\mathbb R}^{n}$ and for every Lebesgue measurable function $F: {\mathbb R}^{n} \rightarrow {\mathbb C}$, we set
$$
\nu(E):=\limsup_{t\rightarrow +\infty}\frac{1}{t^{n/p}}\int_{|{\bf t}|\leq t}\chi_{E}({\bf t})\, d{\bf t} \mbox{ and } \overline{{\mathcal M}}^{E}[F]:=\limsup_{t\rightarrow +\infty}\frac{1}{t^{n/p}}\int_{|{\bf t}|\leq t}F({\bf t})\chi_{E}({\bf t})\, d{\bf t}.
$$
If 
$$
\lim_{t\rightarrow +\infty}\frac{1}{t^{n/p}}\int_{|{\bf t}|\leq t}F({\bf t})\chi_{E}({\bf t})\, d{\bf t}
$$ 
exists in ${\mathbb C},$ then we denote this quantity by ${\mathcal M}^{E}[F].$

Suppose now that the function $F^{(a)}: {\mathbb R}^{n} \rightarrow {\mathbb C}$ is $(a_{1},a_{2},...,a_{n})$-periodic. Then we can find a sequence of infinitely differentiable functions $(\varphi_{k})_{k\in {\mathbb N}}$ with compact support in $S=[0,|a_{1}|] \times ...\times [0,|a_{n}|]$ such that $\varphi_{k} \rightarrow F^{(a)}$ as $k\rightarrow +\infty,$ in $L^{p}(S).$ After that, we extend  $\varphi_{k}(\cdot)$ to an  $(a_{1},a_{2},...,a_{n})$-periodic function $\widetilde{\varphi_{k}}(\cdot)$ defined on the whole space ${\mathbb R}^{n}$ in the usual way. Then it is very simple to prove that 
\begin{align}\label{appr}
\lim_{k\rightarrow +\infty}\limsup_{t\rightarrow +\infty}\Biggl[t^{-(n/p)}\bigl\| F^{(a)}-\widetilde{\varphi_{k}}\bigr\|_{L^{p}(({\mathbb R}^{n})_{t})}\Biggr]=0;
\end{align}
cf. also \cite[p. 483, l. 7-l. 9]{doss}.
Keeping this observation in mind, the following result can be deduced, in a plus-minus technical way, following the argumentation contained in the proofs of \cite[Proposition 1, Proposition 3, Corollary, Lemma 2]{doss} (cf. also Lemma \ref{lemara}, which is needed for the proof of (ii)):

\begin{thm}\label{paramet}
\begin{itemize}
\item[(i)]
Suppose that $p\in [1,\infty),$ $q\in (1,\infty],$ $1/p+1/q=1,$
the function $F : \Lambda \rightarrow Y$ satisfies that $\| F(\cdot)\|_{Y}\in L^{p}(\Lambda_{t})$ for all $t>0$, as well as that $\| F(\cdot +\tau )-F(\cdot)\|_{Y}\in L^{p}(\Lambda_{t})$ for all $t>0$ and $\tau \in \Lambda''.$ Suppose, further, that the function $F(\cdot)$ is Besicovitch-$(p,x,{\rm F}_{1})$-bounded and Besicovitch-$(p,x,{\rm F}_{1},\Lambda'')$-continuous as well as that
condition \emph{(III)} holds, and the set $\Lambda'' \cap {\mathbb Q}^{n}$ is dense in $\Lambda''.$ If
the function $G : \Lambda \rightarrow {\mathbb C}$ satisfies that $ G(\cdot)\in L^{q}(\Lambda_{t})$ for all $t>0,$ and $G(\cdot)$ is Besicovitch-$(q,x,{\rm F}_{2})$-bounded, then for each sequence $(L_{m})_{m\in {\mathbb N}}$ there exists a subsequence $(T_{m})_{m\in {\mathbb N}}$ of $(L_{m})_{m\in {\mathbb N}}$ such that the function
$$
H(\tau):=\lim_{m\rightarrow +\infty}{\rm F}_{1}\bigl( T_{m} \bigr){\rm F}_{2}\bigl( T_{m} \bigr)\int_{\Lambda_{T_{m}}}G({\bf s})F({\bf s}+\tau)\, d{\bf s},\quad \tau \in \Lambda''
$$
is well-defined and bounded.
Furthermore, if the function $F(\cdot)$ is Doss-$(p,x,{\rm F}_{1},\Lambda')$-almost periodic (Doss-$(p,x,{\rm F}_{1},\Lambda')$-uniformly recurrent) and $\Lambda'+\Lambda''\subseteq \Lambda'',$ then the function $H(\cdot)$ is Bohr $\Lambda'$-almost periodic ($\Lambda'$-uniformly recurrent).
\item[(ii)] Suppose that $p\in [1,\infty),$ $q\in (1,\infty],$ $1/p+1/q=1,$ the assumptions in \emph{(i)} hold for the function $F : {\mathbb R}^{n} \rightarrow {\mathbb C},$ with $\Lambda=\Lambda'=\Lambda''={\mathbb R}^{n},$ ${\rm F}_{1}( t)={\rm F}_{2}(t)=t^{-(n/p)},$ $t\in {\mathbb R}^{n},$ and $F(\cdot)$ is Doss-$(p,x,t^{-(n/p)})$-almost periodic. If condition \emph{(A)$_{\infty}$} or \emph{(AS)} holds with $X=\{0\},$ $\phi(x)\equiv x$ and $p(\cdot)\equiv p,$ then, for every real number $\epsilon>0,$ there exists a finite real number $\delta>0$ such that, for every Lebesgue measurable set $E\subseteq {\mathbb R}^{n},$ the assumption $\nu(E)<\delta$ implies $\overline{{\mathcal M}}^{E}[|F|_{Y}^{p}]<\epsilon.$
\item[(iii)] Let the assumptions of \emph{(ii)} hold, and let for each $N>0$ the function $F_{N} : {\mathbb R}^{n} \rightarrow {\mathbb C}$ be defined by $F_{N}({\bf t}):=F({\bf t})$, if $|F({\bf t})|\leq N,$ and $F_{N}({\bf t}):=Ne^{i\arg(F({\bf t}))}$, if $|F({\bf t})|> N.$ Then $\lim_{N\rightarrow +\infty}\overline{{\mathcal M}}^{{\mathbb R}^{n}}[F-F_{N}]=0.$
\item[(iv)] Suppose that the assumptions in \emph{(i)} hold for the function $F : {\mathbb R}^{n} \rightarrow {\mathbb C},$ with $p=1,$ $\Lambda=\Lambda'=\Lambda''={\mathbb R}^{n},$ ${\rm F}_{1}( t)={\rm F}_{2}(t)=t^{-n},$ $t\in {\mathbb R}^{n},$ and $F(\cdot)$ is Doss-$(1,x,t^{-n})$-almost periodic. If condition \emph{(A)$_{\infty}$} or \emph{(AS)} holds with $X=\{0\},$ $\phi(x)\equiv x$ and $p(\cdot)\equiv 1,$ then for each $a\in {\mathbb R}^{n}$ we have that ${\mathcal M}^{{\mathbb R}^{n}}[F]$ exists in ${\mathbb C}$ and ${\mathcal M}^{{\mathbb R}^{n}}[F]={\mathcal M}^{{\mathbb R}^{n}}[F^{(a)}].$
\end{itemize} 
\end{thm}

Before proceeding further, let us note that it is not clear how we can extend the statement (ii) to the vector-valued functions. It is also worth noting that a serious difficulty in our analysis of the multi-dimensional case presents the fact that it is not clear whether we can further generalize the above-mentioned statement by using condition (A) in place of (A)$_{\infty}$ or (AS). Concerning conditions (A), (A)$_{\infty}$, (AS) and the equation \eqref{appr}, we would like to present the following example:

\begin{example}\label{bencina}
\begin{itemize}
\item [(i)] (see also \cite[Example 7.2.2]{nova-selected}) Suppose that $F_{0} : \{ (x,y)\in {\mathbb R}^{2} : 0\leq x+y\leq 2 \} \rightarrow [0,\infty)$ be any continuous function such that the following conditions hold:
\begin{itemize}
\item [(a)] $F_{0}(x,y)=F_{0}(x+1,y+1)$ for every $(x,y)\in {\mathbb R}^{2}$ such that $x+y=0;$
\item[(b)] Let $P_{k}=A_{k}B_{k}C_{k}D_{k}$ be the rectangle in ${\mathbb R}^{2}$ with vertices $A_{k}=(4k-(2/3),(2/3)-4k),$ $B_{k}=(4k-(1/3),(1/3)-4k),$ $C_{k}=(4k+(2/3),(4/3)-4k)$ and $D_{k}=(4k+(1/3),(5/3)-4k),$ 
for each integer $k\in {\mathbb Z}.$ We have $F_{0}(x,y)\geq 2^{|k|}$ for all integers $k\in {\mathbb Z}$ and $(x,y)\in P_{k}.$ 
\end{itemize}
We extend the function $F_{0}(\cdot;\cdot)$ to a continuous $(1,1)$-periodic function $F : {\mathbb R}^{2} \rightarrow [0,\infty)$ in the obvious way. Then the function $F(\cdot,\cdot)$ is not Besicovitch almost periodic since it is not Besicovitch bounded; this follows from the following simple estimates:
\begin{align*}
\int_{\Lambda_{8k\sqrt{2}}}F(x,y)\, dx\, dy \geq 16k\sum_{l=0}^{4k}\frac{\sqrt{2}}{3}2^{l}\geq \frac{16k\sqrt{2}}{3}\bigl(2^{4k+1}-1\bigr),\quad k\in {\mathbb N}.
\end{align*}
Furthermore, there do not exist a finite real constant $M>0$ and an essentially bounded function $G : {\mathbb R}^{2} \rightarrow {\mathbb C}$ such that
\begin{align*}
\limsup_{t\rightarrow +\infty}\Biggl[t^{-2}\int_{|{\bf t}|\leq t}| F({\bf t})-G({\bf t})|\, d{\bf t}\Biggr]\leq M;
\end{align*}
cf. also \eqref{appr}.
\item[(ii)] Set $P(x,y):=e^{i[\sqrt{2}x+y]}+e^{i[2x+y]},$ $x,\ y\in {\mathbb R},$ $a_{1}:=\pi(2+\sqrt{2}),$ $a_{2}:=2\pi(1-\sqrt{2})$ and $a:=(a_{1},a_{2}).$ Then there is no continuous $(a_{j})_{j\in {\mathbb N}_{2}}$-periodic function $F^{(a_{1},a_{2})}(x,y)$ such that \eqref{maleni} holds with $\phi(x)\equiv x,$ $p(\cdot)\equiv 1$ and ${\rm F}(t)\equiv t^{-2}$ ($X=\{0\}$). In actual fact, we have $\sqrt{2}a_{1}+a_{2}=2\pi,$ $2a_{1}+a_{2}=2\pi$ and the validity of \eqref{maleni} would imply
\begin{align}\label{trazimo}
\limsup_{t\rightarrow +\infty}\Biggl[t^{-2}\int_{|(x,y)|\leq t}\Bigl| e^{i[\sqrt{2}x+y]}+e^{i[2x+y]}-F^{(a_{1},a_{2})}(x,y) \Bigr|\, dx\, dy\Biggr]=0.
\end{align}
To see that \eqref{trazimo} cannot be true, it suffices to observe that there exists a real number $\epsilon_{0}:=\min_{(x,y)\in [0,a_{1}] \times [0,-a_{2}]}|e^{i[\sqrt{2}x+y]}+e^{i[2x+y]}-F^{(a_{1},a_{2})}(x,y)|>0$ such that 
$$
\int_{(x,y)\in [0,t]^{2}}\Bigl| e^{i[\sqrt{2}x+y]}+e^{i[2x+y]}-F^{(a_{1},a_{2})}(x,y) \Bigr|\, dx\, dy \geq \frac{t^{2}}{\lfloor a_{1} \rfloor \lfloor |a_{2}| \rfloor}\epsilon_{0}.
$$
\end{itemize}
\end{example}

If $F: \Lambda \times X \rightarrow Y$, $G: \Lambda \times X \rightarrow Y$ and $B\in {\mathcal B},$ then we set
\begin{align*}
{\mathcal M}_{B}^{\phi,{\rm F}}(F,G):=\limsup_{t\rightarrow +\infty}{\rm F}(t)\sup_{x\in B}\Bigl[\phi\Bigl(\bigl\|F({\bf t};x)-G({\bf t};x)\bigr\|_{Y}\Bigr)\Bigr]_{L^{p({\bf t})}(\Lambda_{t})}.
\end{align*}
Then the quantity ${\mathcal M}_{B}^{\phi,{\rm F}}(F,G)$ always exists in $[0,+\infty].$
The subsequent result states that, under certain reasonable assumptions on the functions $\phi(\cdot)$ and ${\rm F}(\cdot),$ any function $F\in e-({\mathcal B},\phi,{\rm F})-B^{p(\cdot)}(\Lambda \times X : Y)$ satisfies condition (A):

\begin{prop}\label{raj}
Suppose that  $F: \Lambda \times X \rightarrow Y$, $F\in e-({\mathcal B},\phi,{\rm F})-B^{p(\cdot)}(\Lambda \times X : Y),$ \emph{(I)-(III)} and the following conditions hold:
\begin{itemize}
\item[(IV)] If $B\in {\mathcal B},$ $a\in \Lambda''$,
and
$(F_{k}: \Lambda \times X \rightarrow Y)_{k\in {\mathbb N}}$ is any sequence of functions which satisfies that $F_{k}({\bf t}+a;x)=F_{k}({\bf t};x),$  ${\bf t}\in \Lambda,$ $x\in B$ and for each $\epsilon>0$ there exists $k_{0}\in {\mathbb N}$ such that ${\mathcal M}_{B}^{\phi,{\rm F}}(F_{k},F_{k'})<\epsilon$ for all integers $k,\ k'\geq k_{0},$ then there exists a function $F: \Lambda \times X \rightarrow Y$ such that 
$F({\bf t}+a;x)=F({\bf t};x),$  ${\bf t}\in \Lambda,$ $x\in B$
and
$\lim_{k\rightarrow +\infty}{\mathcal M}_{B}^{\phi,{\rm F}}(F_{k},F)=0$  (here we assume that for each element $x\in X$ the functions $F_{k}(\cdot;x)$ and $F(\cdot;x)$ are Lebesgue measurable ($k\in {\mathbb N}$)).
\item[(V)] The collection ${\mathcal B}$ consists of bounded subsets of $X$.
\end{itemize}
Then the function $F(\cdot;\cdot)$ satisfies condition \emph{(A).}
\end{prop}

\begin{proof}
We will slightly modify the original argumentation of R. Doss (see \cite[pp. 477-478]{doss}). First of all, we will prove that condition (A) holds for any trigonometric polynomial $P(\cdot;x)=\sum_{l=0}^{m}e^{i \langle \lambda_{l},\cdot \rangle}c_{l}(x),$ $x\in X.$ Let $B\in {\mathcal B}$ and $a\in \Lambda''$ be given; then ${\bf t}+ja\in \Lambda$ for all $j\in {\mathbb N}.$ Define $P^{(a)}_{B}(\cdot;x):=\sum_{l\in  L_{a}}e^{i \langle \lambda_{l},\cdot \rangle}c_{l}(x),$ $x\in X$, where $L_{a}$ denotes the set of all integers $l\in [0,m]$ such that $e^{i\langle \lambda_{l},a \rangle} = 1.$ Clearly, the function $P^{(a)}_{B}(\cdot;x)$ is $a$-periodic for every fixed element $x\in B$. Then a simple computation shows that
\begin{align*}
{\rm F}(t)\sup_{x\in B}&\Biggl[\phi\Biggl(\Biggr\| \frac{1}{k}\sum_{j=0}^{k-1}P({\bf t}+ja;x)-P^{(a)}_{B}({\bf t};x)\Biggr\|_{Y}\Biggr)\Biggr]_{L^{p}(\Lambda_{t})}
\\& ={\rm F}(t)\sup_{x\in B}\Biggl[\phi\Biggl(\Biggr\| \frac{1}{k}\sum_{j=0}^{k-1}
\sum_{l\notin  L_{a}}e^{i \langle \lambda_{l},{\bf t} \rangle}e^{i \langle \lambda_{l},aj \rangle}c_{l}(x)
\Biggr\|_{Y}\Biggr)\Biggr]_{L^{p}(\Lambda_{t})}
\\& ={\rm F}(t)\sup_{x\in B}\Biggl[\phi\Biggl(\Biggr\| \frac{1}{k}
\sum_{l\notin  L_{a}}e^{i \langle \lambda_{l},{\bf t} \rangle}\frac{e^{i \langle \lambda_{l},ak \rangle}-1}{e^{i \langle \lambda_{l},a\rangle}-1}c_{l}(x)
\Biggr\|_{Y}\Biggr)\Biggr]_{L^{p}(\Lambda_{t})},\quad t>0.
\end{align*}
Using the facts that the function $\phi(\cdot)$ is monotonically increasing and the collection ${\mathcal B}$ consists of bounded subsets of $X$, the above computation yields the existence of a finite real constant $c_{B}>0$ such that
\begin{align*}
{\rm F}(t)\sup_{x\in B}&\Biggl[\phi\Biggl(\Biggr\| \frac{1}{k}\sum_{j=0}^{k-1}P({\bf t}+ja;x)-P^{(a)}_{B}({\bf t};x)\Biggr\|_{Y}\Biggr)\Biggr]_{L^{p({\bf t})}(\Lambda_{t})}
\\& \leq \phi\Bigl( \frac{c_{B}}{k} \Bigr)\Bigl[{\rm F}(t)m(\Lambda_{t})^{1/p}\Bigr],\quad t>0.
\end{align*}
Then the required conclusion follows from the continuity of function $\phi(\cdot)$ at zero and the assumption $\limsup_{t\rightarrow +\infty}[{\rm F}(t)m(\Lambda_{t})^{1/p}]<+\infty$. Before proceeding to the general case, let us note that our assumptions on the function $\phi(\cdot),$ the assumption that, for every real number $a>0,$ we have $\limsup_{t\rightarrow +\infty}[{\rm F}(t)/{\rm F}(t+a)]\leq 1,$ and a relatively simple argumentation shows that:
\begin{itemize}
\item[(a)] Define 
\begin{align*}
{\mathcal M}_{B,a}^{\phi,{\rm F}}(F,G):=\limsup_{t\rightarrow +\infty}{\rm F}(t)\sup_{x\in B}\Bigl[\phi\Bigl(\bigl\|F({\bf t}+a;x)-G({\bf t}+a;x)\bigr\|_{Y}\Bigr)\Bigr]_{L^{p}(\Lambda_{t})}.
\end{align*}
Then we have
\begin{align*}
{\mathcal M}_{B,a}^{\phi,{\rm F}}(F,G)\leq {\mathcal M}_{B}^{\phi,{\rm F}}(F,G).
\end{align*}
\item[(b)] ${\mathcal M}_{B}^{\phi,{\rm F}}(dF,dG)\leq \varphi(d){\mathcal M}_{B}^{\phi,{\rm F}}(F,G)$ for all real numbers $d>0.$
\item[(c)] ${\mathcal M}_{B}^{\phi,{\rm F}}(F,G) \leq c[{\mathcal M}_{B}^{\phi,{\rm F}}(F,H)+{\mathcal M}_{B}^{\phi,{\rm F}}(H,G)].$
\item[(d)] ${\mathcal M}_{B}^{\phi,{\rm F}}(F+G,H+W)\leq c[{\mathcal M}_{B}^{\phi,{\rm F}}(F,H)+{\mathcal M}_{B}^{\phi,{\rm F}}(G,W)].$
\end{itemize}
Suppose now that there exists a sequence $(P_{k}(\cdot;\cdot))$ of trigonometric polynomials such that \eqref{jew} holds, i.e.,
$
\lim_{k\rightarrow +\infty}{\mathcal M}_{B}^{\phi,{\rm F}}(F,P_{k})=0.
$
Let $\epsilon>0$ be fixed. Using (a) and (d), we get:{\scriptsize
\begin{align}
&\notag {\mathcal M}_{B}^{\phi,{\rm F}}\Biggl(\frac{F(\cdot;\cdot)+F(\cdot+a;\cdot)+...+F(\cdot+(m-1)a;\cdot)}{m},\frac{P_{k}(\cdot;\cdot)+P_{k}(\cdot+a;\cdot)+...+P_{k}(\cdot+(m-1)a;\cdot)}{m}\Biggr)
\\\label{cvbbvc}& \leq \bigl(c+...+c^{m-1}\bigr)\varphi(1/m){\mathcal M}_{B}^{\phi,{\rm F}}(F,P_{k}) \leq m\varphi(1/m){\mathcal M}_{B}^{\phi,{\rm F}}(F,P_{k})\leq D{\mathcal M}_{B}^{\phi,{\rm F}}(F,P_{k}),\quad k,\ m \in {\mathbb N}.
\end{align}}
On the other hand, using (c) we get the existence of an integer $k_{0}(\epsilon)\equiv k_{0}$ such that:
\begin{align*}
{\mathcal M}_{B}^{\phi,{\rm F}}(P_{k},P_{k'}) \leq 2c\epsilon,\quad k,\ k'\geq k_{0}.
\end{align*}
Keeping in mind \eqref{cvbbvc} and this estimate, we get 
\begin{align}\label{mil}
{\mathcal M}_{B}^{\phi,{\rm F}}\Biggl(\frac{1}{m}\sum_{j=0}^{m-1}P_{k}(\cdot +ja;\cdot),\frac{1}{m}\sum_{j=0}^{m-1}P_{k'}(\cdot +ja;\cdot)\Biggr)
 \leq 2cD\epsilon,\quad k,\ k'\geq k_{0},\ m\in {\mathbb N}.
\end{align}
If  $k,\ k'\geq k_{0},$ then we find the functions $P_{k,B}^{(a)}$ and $P_{k',B}^{(a)}$ from condition (A); then we can use (c) and \eqref{mil} to get the existence of a finite real constant $d>0$ such that ${\mathcal M}_{B}^{\phi,{\rm F}}(P_{k,B}^{(a)},P_{k',B}^{(a)})<d\epsilon$ for all integers $k,\ k'\geq k_{0}.$ Let $F^{(a)}_{B}$ be any function such that
$F^{(a)}_{B}({\bf t}+a;x)=F^{(a)}_{B}({\bf t};x)$ for all ${\bf t}\in \Lambda,$ $x\in B$ and $\lim_{k\rightarrow +\infty}{\mathcal M}_{B}^{\phi,{\rm F}}(P_{k,B}^{(a)},F^{(a)}_{B})=0;$ see (iv). Then the final conclusion follows using \eqref{cvbbvc} and condition (A) for $P_{k}(\cdot;\cdot).$
\end{proof}

\begin{rem}\label{iv}
In connection with condition (IV), we would like to note that the argumentation contained on \cite[p. 478; l. 3-l. 6]{doss} is a bit incorrect because the completeness of $L^{p}$-spaces has been mistakenly used. It is also far from being immediately clear why the function $f^{(a)}(x)$ appearing here must be of period $a.$
\end{rem}
 
If the assumptions (I)-(III) hold, then we have the validity of statements (a)-(d) given in the proof of the afore-mentioned proposition. Keeping in mind this observation, we can 
repeat verbatim the argumentation contained in the proof of \cite[Lemma 1]{doss} to deduce the following:

\begin{prop}\label{dawq}
Suppose that the function $F : \Lambda \times X \rightarrow Y$ satisfies that\\ $\phi(\| F(\cdot;x)\|_{Y})\in L^{p(\cdot)}(\Lambda_{t})$ for all $t>0$ and $x\in X.$ If the assumptions \emph{(I)-(III)} hold and the function $F(\cdot;\cdot)$ satisfies condition \emph{(A)}, then $F(\cdot;\cdot)$ is Besicovitch-$(p,\phi,{\rm F},{\mathcal B})$-bounded.
\end{prop}

In particular, Proposition \ref{dawq} implies that a uniformly recurrent function $f: {\mathbb R} \rightarrow {\mathbb R}$ need not satisfy condition (A); see \eqref{voja} with $\phi(x)\equiv x,$ $p=1$ and ${\rm F}(t)\equiv t^{-1}.$

In the multi-dimensional setting, it seems very plausible that a trigonometric polynomial $P(\cdot)$ does not satisfy condition (AS) in general (see Example \ref{bencina}(ii)), so that we can freely say that the results established in \cite{doss} are primarily intended for the analysis of Besicovitch almost periodic functions of one real variable. We want also to emphasize here that the proof of \cite[Proposition 4]{doss} contains a small gap since the author has not proved that, for any function $f: {\mathbb R}\rightarrow {\mathbb C}$ of class $(D)$ and for any real number $\lambda \in {\mathbb R},$
the function $e^{-i\lambda \cdot}f(\cdot)$ is Doss almost periodic, i.e., satisfies condition \cite[2., p. 477]{doss} with $p=1$ (this seems to be true, but not specifically proved in the paper). Finally, we would like to mention that it is very likely that the statements of \cite[Proposition 4-Proposition 6]{doss} admit extensions to the multi-dimensional setting; hence, it seems very reasonable that a Doss-$p$-almost periodic function $F: {\mathbb R}^{n} \rightarrow {\mathbb C}$ which is Besicovitch-$p$-continuous and satisfies condition (AS) is Besicovitch-$p$-almost periodic ($1\leq p<+\infty$). This is a very unsatisfactory result in the multi-dimensional setting and we will skip all details with regards to this question here. 

\subsection{On Condition (B)}\label{bea} 

In this subsection, we will consider the following condition:

\begin{itemize}  
\item[(B)] Let $\Omega=[0,1]^{n},$ $l\Omega \subseteq \Lambda$ and $\Lambda +l\Omega \subseteq \Lambda$ for all $l>0.$ If ${\rm F}_{1} : (0,\infty) \rightarrow (0,\infty),$ ${\rm F} : (0,\infty) \rightarrow (0,\infty)$ and $p\in {\mathcal P}(\Lambda),$ then condition (B) means that, for every $\lambda \in {\mathbb R}^{n},$ we have:
\begin{align}\label{mica}
\lim_{l\rightarrow +\infty}{\rm F}_{1}(l)\limsup_{t\rightarrow +\infty}{\rm F}(t)\sup_{x\in B}\Biggl\| \Biggl[ \int_{y+l\Omega} -\int_{l\Omega}\Biggr] e^{i\lambda {\bf t}}F({\bf t};x) \, d{\bf t}  \Biggr\|_{L^{p(y)}(\Lambda_{t}: Y)}=0.
\end{align}
\end{itemize}
Let us note that, in the original analysis of R. Doss \cite{doss1}, we have $p(\cdot)\equiv 1,$ 
${\rm F}_{1}(l)\equiv 1/l,$
${\rm F}(t)\equiv 1/t$, $\Lambda = {\mathbb R}$ and $X=\{0\}.$ The situation in which the equation \eqref{mica} holds for all values of $\lambda \in {\mathbb R}^{n} \setminus \{ \lambda_{0} \}$ but not for the exactly one value $\lambda=\lambda_{0}\in {\mathbb R}^{n}$ is possible; for example, in the one-dimensional setting, we know that the function $F : {\mathbb R} \rightarrow {\mathbb C}$, given by
$F(t):=e^{-i\lambda_{0}t},$ $t\geq 0$ and $F(t):=-e^{-i\lambda_{0}t},$ $t< 0,$ satisfies \eqref{mica} for all values $\lambda \in {\mathbb R} \setminus \{ \lambda_{0} \}$ but not for $\lambda_{0}$ (\cite{doss1}).

Concerning condition (B), we will first clarify the following result for the multivariate trigonometric polynomials:

\begin{prop}\label{zaed}
Suppose that $\Omega =[0,1]^{n},$ $\Lambda={\mathbb R}^{n}$ and $\lim_{l\rightarrow +\infty}[l^{n-1}{\rm F}_{1}(l)]=0.$ If the collection ${\mathcal B}$ consists solely of bounded subsets of $X$ and conditions \emph{(I)-(III)} hold, then condition \emph{(B)} holds for any trigonometric polynomial $P(\cdot;\cdot).$
\end{prop}

\begin {proof}
Let $P({\bf t};x)=\sum_{s=0}^{m}e^{i\langle \lambda_{s},{\bf t} \rangle}c_{s}(x)$ for some continuous functions $c_{s}(\cdot).$ It suffices to show that
$$
\lim_{l\rightarrow +\infty}{\rm F}_{1}(l)\limsup_{t\rightarrow +\infty}{\rm F}(t)\sup_{x\in B}\Biggl\| \Biggl[ \int_{y+l\Omega} -\int_{l\Omega}\Biggr] P({\bf t};x) \, d{\bf t}  \Biggr\|_{L^{p}(\Lambda_{t}: Y)}=0.
$$
Let $k_{s}$ denote the number of all non-zero components of vector $\lambda_{s}=(\lambda_{s}^{1},...,\lambda_{s}^{n}).$ If $k_{s}=n,$ then the term $e^{i\langle \lambda_{s},{\bf t} \rangle}c_{s}(x)$ is  meaningless after the integration over the cubes $y+l\Omega $ and  $l\Omega.$ Because of that, we may assume without loss of generality that $k_{s}\leq n-1$ for all $s\in \{0,1,...,m\}$ as well as that $\lambda_{s}^{j}=0$ for all $s\in \{0,1,...,m\}$ and $j\in \{1,...,k\},$ where $k\leq n-1.$
Therefore, we need to prove that
\begin{align*}
\lim_{l\rightarrow +\infty}&{\rm F}_{1}(l)\limsup_{t\rightarrow +\infty}{\rm F}(t)\sup_{x\in B}\Biggl\| \Biggl[ \int_{y+l\Omega} -\int_{l\Omega}\Biggr]
\\& \times \sum_{s=0}^{m}\sum_{j=k+1}^{n}e^{i[\lambda_{s}^{k+1}t_{k+1}+...+\lambda_{s}^{n}t_{n}]}c_{s}(x) \, dt_{1}\, dt_{2}\, ... \, dt_{n}  \Biggr\|_{L^{p}(\Lambda_{t}: Y)}=0, 
\end{align*}
i.e., that
\begin{align*}
\lim_{l\rightarrow +\infty}&{\rm F}_{1}(l)l^{k}\limsup_{t\rightarrow +\infty}{\rm F}(t)\sup_{x\in B}\Biggl\| \Biggl[ \int_{y+l\Omega} -\int_{l\Omega}\Biggr] \\& \times \sum_{s=0}^{m}\sum_{j=k+1}^{n}e^{i[\lambda_{s}^{k+1}t_{k+1}+...+\lambda_{s}^{n}t_{n}]}c_{s}(x) \, dt_{k+1}\, dt_{k+2}\, ... \, dt_{n}  \Biggr\|_{L^{p}(\Lambda_{t}: Y)}=0.
\end{align*}
It is clear that we have the existence of a finite real constant $c>0$ such that, for every $y\in \Lambda_{t}$ and $l>0,$ we have:
$$
\Biggl\|\Biggl[\int_{y+l\Omega} -\int_{l\Omega}\Biggr]  \sum_{s=0}^{m}\sum_{j=k+1}^{n}e^{i[\lambda_{s}^{k+1}t_{k+1}+...+\lambda_{s}^{n}t_{n}]}c_{s}(x) \, dt_{k+1}\, dt_{k+2}\, ... \, dt_{n}\Biggr\|_{Y} \leq c.
$$
Since $\lim_{l\rightarrow +\infty}[l^{n-1}{\rm F}_{1}(l)]=0$ and $\limsup_{t\rightarrow +\infty}{\rm F}(t)[m(\Lambda_{t})]^{1/p}<+\infty,$ this simply completes the proof.
\end{proof}

In the continuation of subsection, we will first consider the one-dimensional setting, show that the necessity in \cite[Theorem 1]{doss1} can be extended to the vector-valued Besicovitch-$p$-almost periodic functions, where $1\leq p<\infty,$ and emphasize some difficulties in proving the sufficiency in this theorem in the case that $p>1$ (in the existing literature, we have found many open problems regarding Besicovitch-$p$-almost periodic functions with the exponent $p>1$). Keeping in mind the proof of the above-mentioned theorem, which works in the vector-valued case, as well as the statements of Proposition \ref{anat}(ii) and Proposition \ref{raj-sed}, it suffices to show that condition (B) holds for Besicovitch-$p$-almost periodic functions $F : {\mathbb R}\rightarrow Y,$ with $1< p<\infty,$ $\lambda=0,$ ${\rm F}_{1}(l)\equiv 1/l$ and ${\rm F}_{1}(t)\equiv 1/t^{1/p}.$ Further on, keeping in mind Proposition \ref{zaed-multi} below and its proof (condition \eqref{bila} holds on account of the H\"older inequality), it suffices to show that there exists a finite real constant $c>0$ such that, for every real number $\epsilon \in (0,1),$ there exists a real number $l_{0}>0$ such that the assumptions $l\geq l_{0}$ and ${\mathcal M}^{x,{\rm F}}(F,P)<\epsilon$ imply the existence of a sufficiently large number $t_{l}>0$ such that, for every $t\geq t_{l},$ we have
\begin{align*}
l^{-1}t^{-(1/p)}\Biggl( \int_{-t}^{t} \Biggl[\int_{y}^{y+l}\bigl\|F(s) -P(s)\bigr\|_{Y} \, ds\Biggr]^{p}\, dy  \Biggr)^{1/p}<c\epsilon.
\end{align*}
Let $\epsilon>0$ be given, let $l_{0}=1$ and $l\geq 1.$ We will show that we can take $c=5$ in the above requirement. First of all, we have the existence of a finite real number $t_{0}>0$ such that
\begin{align*}
\int_{-t}^{t}\bigl\|F(s) -P(s)\bigr\|_{Y}^{p} \, ds \leq \epsilon t^{p},\quad t\geq t_{0}.
\end{align*}
Then there exists a sufficiently large number $t_{l}^{1}>0$ such that
\begin{align*}
l^{-1}t^{-(1/p)}\Biggl( \int_{-(t_{0}+l)}^{t_{0}+l} \Biggl[\int_{y}^{y+l}\bigl\|F(s) -P(s)\bigr\|_{Y} \, ds\Biggr]^{p}\, dy  \Biggr)^{1/p}<\epsilon,\quad t\geq t_{l}^{1}.
\end{align*}
There exists a great similarity in the analysis of estimates for the intervals $[-t,-(t_{0}+l)]$ and $[t_{0}+l,t]$ and, because of that, we will only prove that there exists a sufficiently large number $t_{l}^{2}>t_{0}$ such that
\begin{align*}
l^{-1}t^{-(1/p)}\Biggl( \int_{t_{0}+l}^{t} \Biggl[\int_{y}^{y+l}\bigl\|F(s) -P(s)\bigr\|_{Y} \, ds\Biggr]^{p}\, dy  \Biggr)^{1/p}<2\epsilon,\quad t\geq t_{l}^{2}.
\end{align*}
This follows from the next computation (at the fourth line, we can use the inequality appearing on l. 12, p. 134 of \cite{doss1}, which is a consequence of a simple computation with double integrals):
\begin{align*}
l^{-1}& t^{-(1/p)}\Biggl( \int_{t_{0}+l}^{t} \Biggl[\int_{y}^{y+l}\bigl\|F(s) -P(s)\bigr\|_{Y} \, ds\Biggr]^{p}\, dy  \Biggr)^{1/p}
\\& \leq l^{-1} t^{-(1/p)}\Biggl( \int_{t_{0}+l}^{t} \Biggl[l^{1-\frac{1}{p}}\Biggl(\int_{y}^{y+l}\bigl\|F(s) -P(s)\bigr\|_{Y}^{p} \, ds\Biggr)^{1/p}\Biggr]^{p}\, dy  \Biggr)^{1/p}
\\& =l^{-(1/p)} t^{-(1/p)}\Biggl( \int_{t_{0}+l}^{t} \int_{y}^{y+l}\bigl\|F(s) -P(s)\bigr\|_{Y}^{p} \, ds\, dy  \Biggr)^{1/p}
\\&\leq l^{-(1/p)} t^{-(1/p)}\Biggl( l \int_{-t}^{t+l} \bigl\|F(s) -P(s)\bigr\|_{Y}^{p} \, ds  \Biggr)^{1/p}
\\& \leq \epsilon t^{-(1/p)}[2\bigl(t+l \bigr)]^{1/p}
<2\epsilon,\quad t\geq t_{l}^{2}.
\end{align*}

If $p>1,$ then it is very difficult to show that the validity of condition (B) with ${\rm F}_{1}(l)\equiv l^{-1}$ and ${\rm F}(t)\equiv t^{-(1/p)}$ for a Besicovitch-$p$-continuous function $F: {\mathbb R} \rightarrow Y$ implies the validity of condition (A) for $F(\cdot),$ even in the scalar-valued case. In actual fact, it is very simple to prove that (B) implies the validity of equation obtained by replacing the term $|\cdot|$ in the equation \cite[(4)]{doss1}
with the term $|\cdot|^{p}.$ But, if we replace the term (we will consider the scalar-valued case, only)
$$
\Biggl| c^{-1}\int^{c}_{0}\Biggl[ n^{-1}\sum_{k=0}^{n-1}f(t+x+kc)-n^{-1}\sum_{k=0}^{n-1}f(x+kc)\Biggr]K_{m}(t)\, dt \Biggr|
$$
in the equation \cite[(*), p. 136]{doss1} with the term
\begin{align*}
\Biggl|  c^{-1}\int^{c}_{0}\Biggl[ n^{-1}\sum_{k=0}^{n-1}f(t+x+kc)-n^{-1}\sum_{k=0}^{n-1}f(x+kc)\Biggr]K_{m}(t)\, dt \Biggr|^{p},
\end{align*}
which can be majorized by
\begin{align*}
\leq c^{-p}\int^{c}_{0}\Biggl| n^{-1}\sum_{k=0}^{n-1}f(t+x+kc)-n^{-1}\sum_{k=0}^{n-1}f(x+kc)\Biggr |^{p}\, dt \cdot \Biggl( \int^{c}_{0}K_{m}^{q}(t)\, dt \Biggr)^{p/q}
\end{align*}
with the help of the H\"older inequality, then it is impossible to control the term $( \int^{c}_{0}K_{m}^{q}(t) \, dt)^{p/q}$ as $m\rightarrow +\infty.$ This can be done only in the case that $p=1$ because the 
Fej\' er kernels 
$$
K_{m}(t)=m^{-1}\sin^{2}(m\pi t/c)/\sin^{2}(\pi t/c),\quad t\in {\mathbb R}\ \ (m\in {\mathbb N})
$$ 
are uniformly integrable in a neighborhood of zero with respect to $m\in {\mathbb N}$ but not uniformly $q$-integrable in a neighborhood of zero with respect to $m\in {\mathbb N},$ if $q>1.$

Concerning the multi-dimensional setting, we will prove the following result:

\begin{prop}\label{zaed-multi}
Let $\Omega =[0,1]^{n},$ $\Lambda={\mathbb R}^{n}$ or $\Lambda=[0,\infty)^{n},$ $\lim_{l\rightarrow +\infty}[l^{n-1}{\rm F}_{1}(l)]=0,$ and let
\begin{align}\label{bila}
\lim_{l\rightarrow +\infty}{\rm F}_{1}(l)\sup_{x\in B} \int_{l\Omega} \Bigl\| F({\bf t};x) -P({\bf t};x) \Bigr\|_{Y}\, d{\bf t} =0,\quad B\in {\mathcal B}.
\end{align}
If there exists a finite real constant $c>0$ such that, for every real number $\epsilon \in (0,1)$ and for every set $B\in {\mathcal B},$ there exists a real number $l_{0}>0$ such that the assumptions $l\geq l_{0}$ and ${\mathcal M}_{B}^{x,{\rm F}}(F,P)<\epsilon$ imply the existence of a sufficiently large number $t_{l}>0$ such that, for every $t\geq t_{l},$ we have
\begin{align}\label{bjesovi}
{\rm F}_{1}(l){\rm F}(t)\sup_{x\in B}\Biggl\|  \int_{y+l\Omega} \Bigl(F({\bf t};x) -P({\bf t};x) \Bigr)\, d{\bf t}  \Biggr\|_{L^{p}(\Lambda_{t}: Y)}<c\epsilon,
\end{align}
the collection ${\mathcal B}$ consists solely of bounded subsets of $X$ and conditions \emph{(I)-(III)} hold, then condition \emph{(B)} holds for any function $F\in e-({\mathcal B},{\rm F})-B^{p}(\Lambda \times X : Y).$
\end{prop}

\begin {proof}
Let $\epsilon>0$ and $B\in {\mathcal B}$ be given; we will consider the value $\lambda=0$ in (B), only. It is clear that there exist a trigonometric polynomial $P(\cdot;\cdot)$ and a finite real number
$t_{0}>0$ such that, for every real number $t\geq t_{0},$
we have
\begin{align*}
{\rm F}(t)\sup_{x\in B}\bigl\|F({\bf t};x)-P({\bf t};x)\bigr\|_{L^{p}(\Lambda_{t} : Y)}<\epsilon.
\end{align*}
Then we have ($l>0$):
\begin{align*}
{\rm F}_{1}(l)&\limsup_{t\rightarrow +\infty}{\rm F}(t)\sup_{x\in B}\Biggl\| \Biggl[ \int_{y+l\Omega} -\int_{l\Omega}\Biggr]F({\bf t};x) \, d{\bf t}  \Biggr\|_{L^{p}(\Lambda_{t}: Y)}
\\& \leq {\rm F}_{1}(l)\limsup_{t\rightarrow +\infty}{\rm F}(t)\sup_{x\in B}\Biggl\| \Biggl[ \int_{y+l\Omega} -\int_{l\Omega}\Biggr] \Bigl(F({\bf t};x) -P({\bf t};x) \Bigr)\, d{\bf t}  \Biggr\|_{L^{p}(\Lambda_{t}: Y)}
\\& +{\rm F}_{1}(l)\limsup_{t\rightarrow +\infty}{\rm F}(t)\sup_{x\in B}\Biggl\| \Biggl[ \int_{y+l\Omega} -\int_{l\Omega}\Biggr]P({\bf t};x) \, d{\bf t}  \Biggr\|_{L^{p}(\Lambda_{t}: Y)}
\\& \leq {\rm F}_{1}(l)\limsup_{t\rightarrow +\infty}{\rm F}(t)\sup_{x\in B}\Biggl\|  \int_{y+l\Omega} \Bigl(F({\bf t};x) -P({\bf t};x) \Bigr)\, d{\bf t}  \Biggr\|_{L^{p}(\Lambda_{t}: Y)}
\\& + {\rm F}_{1}(l)\limsup_{t\rightarrow +\infty}{\rm F}(t)\sup_{x\in B}\Biggl\|  \int_{l\Omega} \Bigl(F({\bf t};x) -P({\bf t};x) \Bigr)\, d{\bf t}  \Biggr\|_{L^{p}(\Lambda_{t}: Y)}
\\& +{\rm F}_{1}(l)\limsup_{t\rightarrow +\infty}{\rm F}(t)\sup_{x\in B}\Biggl\| \Biggl[ \int_{y+l\Omega} -\int_{l\Omega}\Biggr]P({\bf t};x) \, d{\bf t}  \Biggr\|_{L^{p}(\Lambda_{t}: Y)}.
\end{align*}
The third term in the last estimate can be majorized as in Proposition \ref{zaed}. For the second term, we can use the assumptions \eqref{bila} and $\limsup_{t\rightarrow +\infty}{\rm F}(t)[m(\Lambda_{t})]^{1/p}<+\infty.$ For the first term, we can use our assumption \eqref{bjesovi}.
\end{proof}

It is not difficult to prove that, for every locally integrable function $F : [0,\infty)^{n} \rightarrow Y,$ $t>0$ and $t_{0}\in (0,t),$  we have the estimate
$$
\int_{\Omega_{0}}\int_{y+l\Omega}\|F({\bf s})\|_{Y}\, d{\bf s}\, dy\leq \int_{(t+l)\Omega}\|F({\bf s})\|_{Y}\, d{\bf s},
$$
where $\Omega_{0}:=t\Omega \setminus [0,t_{0}]^{n}.$ Using the H\"older inequality and repeating verbatim the argumentation given in the one-dimensional setting, we can prove that the requirements of Proposition \ref{zaed-multi} hold with ${\rm F}_{1}(l)\equiv l^{-n}$ and ${\rm F}(t)\equiv t^{-(n/p)}.$

Finally, it seems reasonable to ask whether the validity of condition (B) with $p(\cdot)\equiv 1$ implies, along with the corresponding Besicovitch continuity assumption, the validity of condition (A) in the multi-dimensional setting. We will not consider this question here as well as certain possibilities to extend the results established by A. S. Kovanko \cite{kovanko}-\cite{kovanko1} to the multi-dimensional setting.

\section{Applications to the abstract Volterra integro-differential equations}\label{convinv}

In this section, we will provide several applications of our results to the various classes of abstract Volterra integro-differential equations and the partial differential equations.\vspace{0.1cm}

1. In this part, we will first prove a new result about the invariance of Besicovitch-$p$-almost periodicity under the actions of infinite convolution product
\begin{align}\label{trigpol}
t\mapsto F(t):=\int^{t}_{-\infty}R(t-s)f(s)\, ds,\quad t\in {\mathbb R};
\end{align}
we will only note here, without going into full details, that this result can be formulated in the multi-dimensional setting as well (\cite{nova-selected}). We assume that
the operator family $(R(t))_{t>0}\subseteq L(X,Y)$ satisfies that there exist finite real constants $M>0$, $\beta \in (0,1]$ and $\gamma >1$ such that
\begin{align}\label{rad}
\bigl\| R(t)\bigr\|_{L(X,Y)}\leq M\frac{t^{\beta-1}}{1+t^{\gamma}},\quad t>0,
\end{align}
and $f(\cdot)$ is Besicovitch-$p$-almost periodic.

The following result is closely connected with the statements of \cite[Theorem 2.11.4, Theorem 2.13.10, Theorem 2.13.12]{nova-mono}:

\begin{prop}\label{stan}
Suppose that the operator family $(R(t))_{t>0}\subseteq L(X,Y)$ satisfies \eqref{rad}, as well as that  $a>0,$ $\alpha>0,$ $1\leq p<+\infty,$ $\alpha p \geq 1,$ $ap\geq 1,$
$\alpha p(\beta-1)/(\alpha p-1)>-1$ if $\alpha p>1$, and $\beta=1$ if $\alpha p=1.$ If the function $f : {\mathbb R} \rightarrow X$ is Stepanov-$(\alpha p)$-bounded, i.e.,
$$
\bigl\|f\bigr\|_{S^{p}}:=\sup_{t\in {\mathbb R}}\int^{t+1}_{t}\bigl\|f(s)\bigr\|^{\alpha p}\, ds<+\infty,
$$
and $f\in e-(x^{\alpha},t^{-a})-B^{p}({\mathbb R}: X),$
then the function $F(\cdot)$, given by \eqref{trigpol}, is bounded, continuous and belongs to the class $e-(x^{\alpha},t^{-a})-B^{p}({\mathbb R}: Y).$
\end{prop}

\begin{proof}
Arguing as in the proof of \cite[Proposition 2.6.11]{nova-mono}, we may conclude that the function $F(\cdot)$ is well-defined, bounded and continuous. Let $(P_{k})$ be a sequence of trigonometric polynomials such that
\begin{align}\label{dejan1}
\lim_{k\rightarrow +\infty}\limsup_{t\rightarrow +\infty}\frac{1}{2t^{ap}}\int^{t}_{-t}\bigl\| f(s)-P_{k}(s)\bigr\|^{\alpha p}\, ds=0.
\end{align}
Applying again \cite[Proposition 2.6.11]{nova-mono}, we get that the function $t\mapsto F_{k}(t)\equiv \int^{t}_{-\infty}R(t-s)P_{k}(s)\, ds,$ $t\in {\mathbb R}$ is almost periodic and we only need to prove that
\begin{align}\label{dejan}
\lim_{k\rightarrow +\infty}\limsup_{t\rightarrow +\infty}\frac{1}{2t^{ap}}\int^{t}_{-t}\bigl\| F(s)-F_{k}(s)\bigr\|^{\alpha p}\, ds=0.
\end{align}
In the remainder of the proof, we will consider case $\alpha p>1$ since the consideration is quite similar if $\alpha p=1.$ Let $\zeta \in (1/(\alpha p),(1/(\alpha p))+\gamma-\beta).$ Then it is clear that the function
$s\mapsto |s|^{\beta-1}(1+|s|)^{\zeta}/(1+|s|^{\gamma}),$
$s\in {\mathbb R}$ belongs to the space $L^{\alpha p/(\alpha p-1)}((-\infty,0));$ further on, arguing as in the proof of \cite[Theorem 2.11.4]{nova-mono}, we have that the function $s\mapsto (1+|s|)^{-\zeta}\| P_{k}(s+z)-f(s+z)\|,$
$s\in {\mathbb R}$ belongs to the space $L^{\alpha p}((-\infty,0))$ for all $k\in {\mathbb N}$ and $z\in {\mathbb R}.$ The estimate \eqref{dejan} follows from the next computation ($M_{1}>0$ and $c>0$ are finite real constants):
\begin{align*}
\frac{1}{2t^{ap}}&\int^{t}_{-t}\bigl\| F(s)-F_{k}(s)\bigr\|^{\alpha p}\, ds  
\\& \leq \frac{1}{2t^{ap}}\int^{t}_{-t}\Biggl|\int^{0}_{-\infty}\|R(-z)\| \cdot \bigl\| P_{k}(s+z)-f(s+z)\bigr\| \, dz \Biggr|^{\alpha p}\, ds
\\& \leq \frac{M}{2t^{ap}}\int^{t}_{-t}\Biggl|\int^{0}_{-\infty}
\frac{|z|^{\beta-1}(1+|z|)^{\zeta}}{(1+|z|^{\gamma})}\cdot (1+|z|)^{ -\zeta}\bigl\| P_{k}(s+z)-f(s+z)\bigr\|\, dz \Biggr|^{\alpha p}\, ds
\\& \leq \frac{M_{1}}{2t^{ap}}\int^{t}_{-t}\int^{0}_{-\infty}\frac{1}{(1+|z|^{\alpha \zeta})^{p}}\bigl\| P_{k}(s+z)-f(s+z)\bigr\|^{\alpha p}\, dz \, ds
\\& =\frac{M_{1}}{2t^{ap}}\int^{t}_{-t}\int^{t}_{z}\frac{1}{(1+|z-s|^{\alpha \zeta})^{p}}\bigl\| P_{k}(z)-f(z)\bigr\|^{\alpha p}\, ds \, dz
\\& +\frac{M_{1}}{2t^{ap}}\int^{t}_{-\infty}\int^{t}_{-t}\frac{1}{(1+|z-s|^{\alpha \zeta})^{p}}\bigl\| P_{k}(z)-f(z)\bigr\|^{\alpha p}\, ds \, dz
\\& \leq \frac{M_{1}}{t^{ap}}\int^{t}_{-t}\bigl\| P_{k}(z)-f(z)\bigr\|^{\alpha p} \, dz \cdot \int^{+\infty}_{-\infty}\frac{ds}{(1+|s|^{\zeta})^{\alpha p}}
\\& +\frac{M_{1}}{2t^{ap}}\int^{-3t}_{-\infty}\int^{t}_{-t}\frac{1}{(1+|z-s|^{\alpha \zeta})^{p}}\bigl\| P_{k}(z)-f(z)\bigr\|^{\alpha p}\, ds \, dz
\\& +\frac{M_{1}}{2t^{ap}}\int^{3t}_{-3t}\int^{t}_{-t}\frac{1}{(1+|z-s|^{\alpha \zeta})^{p}}\bigl\| P_{k}(z)-f(z)\bigr\|^{\alpha p}\, ds \, dz
\\& \leq  \frac{M_{1}}{t^{ap}}\int^{t}_{-t}\bigl\| P_{k}(z)-f(z)\bigr\|^{\alpha p} \, dz \cdot \int^{+\infty}_{-\infty}\frac{ds}{(1+|s|^{\zeta})^{\alpha p}}
\\& + \frac{M_{1}}{t^{ap}}\int^{3t}_{-3t}\bigl\| P_{k}(z)-f(z)\bigr\|^{\alpha p} \, dz \cdot \int^{+\infty}_{-\infty}\frac{ds}{(1+|s|^{\zeta})^{\alpha p}}
\\& +\frac{cM_{1}t}{2t^{ap}}\int^{-3t}_{-\infty}\frac{1}{(1+|z/2|^{\alpha \zeta})^{p}}\bigl\| P_{k}(z)-f(z)\bigr\|^{\alpha p} \, dz,
\end{align*}
involving the H\"older inequality, the Fubini theorem and an elementary change of variables in the double integral; here we use \eqref{dejan1} and the fact that
$$
\lim_{t\rightarrow +\infty}\int^{-3t}_{-\infty}\frac{1}{(1+|z/2|^{\alpha \zeta})^{p}}\bigl\| P_{k}(z)-f(z)\bigr\|^{\alpha p} \, dz=0.
$$
\end{proof}

In the following result, the inhomogeneity $f(\cdot)$ is not necessarily Stepanov-$(\alpha p)$-bounded:

\begin{prop}\label{stan1}
Suppose that the operator family $(R(t))_{t>0}\subseteq L(X,Y)$ satisfies \eqref{rad}, as well as that $a>0,$ $\alpha>0,$ $1\leq p<+\infty,$ $\alpha p \geq 1,$ $ap\geq 1,$
$\alpha p(\beta-1)/(\alpha p-1)>-1$ if $\alpha p>1$, and $\beta=1$ if $\alpha p=1.$ If the function $f : {\mathbb R} \rightarrow X$ belongs to the class $e-(x^{\alpha},t^{-a})-B^{p}({\mathbb R}: X)$ and there exists a finite real constant $M>0$ such that $\|f(t)\|\leq M(1+|t|)^{b},$ $t\in {\mathbb R}$ for some real constant $b \in [0,\gamma-\beta),$
then the function $F(\cdot)$, given by \eqref{trigpol}, is continuous, belongs to the class $e-(x^{\alpha},t^{-a})-B^{p}({\mathbb R}: Y),$ and there exists a finite real constant $M'>0$ such that $\|F(t)\|_{Y}\leq M'(1+|t|)^{b},$ $t\in {\mathbb R}.$
\end{prop}

\begin{proof}
The proof is very similar to the proof of Proposition \ref{stan} and we will only outline the main details. 
Since $b \in [0,\gamma-\beta),$ it can be simply shown that the function $F(\cdot)$ is well-defined, measurable as well as that there exists a finite real constant $M'>0$ such that $\|F(t)\|_{Y}\leq M'(1+|t|)^{b},$ $t\in {\mathbb R}.$ In order to prove the continuity of function $F(\cdot)$ at the fixed point $t\in {\mathbb R},$
we take first any integer $k\in {\mathbb N}$ such that 
\begin{align}\label{stan2}
\int_{k}^{+\infty}\frac{s^{\beta-1}(1+s)^{b}}{1+s^{\gamma}}\, ds \leq \frac{\epsilon}{4M(1+|t|^{b})3^{b}},
\end{align}
where a real number $\epsilon>0$ is given in advance and the real constant $M>0$ has the same value as in \eqref{rad}. Since $\|R(\cdot)\|_{L(X,Y)} \in L^{\alpha p/(\alpha p-1)}[0,k]$ and $\|f\|\in L^{\alpha p}_{loc}[0,k],$ we can apply the H\"older inequality in order to see that the function $F_{k}(\cdot):=\int^{k}_{0}R(s)f(\cdot-s)\, ds,$ $\cdot\in {\mathbb R}$ is continuous. Take any $\delta\in (0,1)$ such that 
\begin{align}\label{zadar}
\Biggl\| \int^{k}_{0}R(s)\bigl[f(t-s)-f(t'-s)\bigr]\, ds\Biggr\|_{Y}<\frac{\epsilon}{2} \ \ \mbox{ for }\ \  |t-t'|\leq \delta.
\end{align}
A very simple argumentation involving \eqref{stan2} shows that
$$
\Biggl\|\int_{k}^{+\infty}R(s)\bigl[f(t-s)-f(t'-s)\bigr]\, ds \Biggr\|_{Y}\leq \epsilon/2,
$$
which along with \eqref{zadar} completes the proof of continuity of function $F(\cdot)$ at the point $t.$ The belonging of function $F(\cdot) $ to the class $e-(x^{\alpha},t^{-1/p})-B^{p}({\mathbb R}: Y)$ can be proved as above, by taking any real number $\zeta \in ((1/ (\alpha p))+b,(1/ (\alpha p))+\gamma-\beta).$
\end{proof}

\begin{rem}\label{split}
Suppose that there exist finite real constants $M>0,$ $c>0$ and $\beta \in (0,1]$ such that
\begin{align*}
\bigl\| R(t)\bigr\|_{L(X,Y)}\leq Me^{-ct}t^{\beta-1},\quad t>0.
\end{align*}
This estimate appears in the analysis of (degenerate) semigroups of operators satisfying condition \cite[(P)]{nova-mono} and, in particular, in the analysis of semigroups of operators generated by almost sectorial operators. In this case, we do not need the assumption $\alpha <\gamma-\beta. $
\end{rem}

It is clear that Proposition \ref{stan} and Proposition \ref{stan1} can be applied to a large class of the abstract (degenerate) Volterra integro-differential equations without initial conditions. For example, we can apply this result in the analysis of the existence and uniqueness of Besicovitch-$p$-almost periodic type solutions of 
the initial value problems with constant coefficients
\[
\begin{array}{l}
D_{t,+}^{\gamma}u(t,x)=\sum_{|\alpha|\leq k}a_{\alpha}f^{(\alpha)}(t,x)+f(t,x),\ t\in {\mathbb R},\ x\in \mathbb{R}^n
\end{array}
\] 
in the space $L^{p}(\mathbb{R}^n),$ where 
$\gamma \in (0,1),$
$D_{t,+}^{\gamma}u(t)$ denotes the Weyl-Liouville fractional derivative of order $\gamma,$
$1\leq p<\infty $
and some extra assumptions are satisfied. We can also consider the existence and uniqueness of Besicovitch-$p$-almost periodic type solutions of the fractional Poisson heat equation in $L^{p}(\mathbb{R}^n),$ and a class of the abstract fractional differential equations with the higher-order elliptic operators in the H\"older spaces (\cite{nova-mono}).\vspace{0.1cm}

2. In this part, we analyze the abstract nonautonomous
differential equations of first order. First of all,
we will remind the readers of some basic definitions from the theory of evolution equations, hyperbolic evolution systems and Green's functions (see \cite{schnaubelt} and the references cited in \cite[Section 2.14]{nova-mono}). \index{hyperbolic evolution system} \index{Green's function}

\begin{defn}\label{paz}
A family $\{U(t, s) : t \geq s,\  t,\  s  \in {\mathbb R}\}$ of bounded linear operators on $X$ is said to be
an evolution system if and only if the following holds:\index{evolution system}
\begin{itemize}
\item[(i)] $U(s, s) = I,$ $U(t, s) = U(t, r)U(r, s)$ for $t \geq  r \geq  s$ and $t,\ r,\ s  \in {\mathbb R},$
\item[(ii)] $\{(\tau, s) \in {\mathbb R}^{2} : \tau > s \} \ni (t, s) \mapsto U(t, s)x$ is continuous for any fixed element $x\in X$.
\end{itemize}
\end{defn}

Here, $I$ denotes the identity operator on $X.$ We assume that the family $A(\cdot)$
satisfies conditions \cite[(H1)-(H2)]{nova-mono}; then there exists an evolution system $\{U(t, s) : t \geq s,\  t,\  s  \in {\mathbb R}\}$ generated by the family $A(t).$
If  $\Gamma(\cdot,\cdot)$ denotes the associated Green's function, then we know that there exists a finite real constant $M>0$ such that
\begin{align}\label{srq}
\|\Gamma(t,s)\| \leq Me^{-\omega |t-s| },\quad t,\ s\in {\mathbb R}.
\end{align}
The function
\begin{align}\label{berlinstrase}
u(t):=\int^{+\infty}_{-\infty}\Gamma(t,s)f(s)\, ds,\quad t\in {\mathbb R}
\end{align}
is said to be a unique mild solution of the abstract Cauchy problem
\begin{align*}
u^{\prime}(t)=A(t)u(t)+f(t),\quad t\in {\mathbb R}.
\end{align*}
The proof of subsequent theorem can be deduced using the argumentation employed in the proof of \cite[Theorem 3.7.1]{nova-mono}, where we have assumed that $p=1,$ the argumentation contained in the proof of Proposition \ref{stan}, the estimate \eqref{srq}, and the following facts:
\begin{itemize}
\item[(i)] If $P(\cdot)$ is a trigonometric polynomial, then the functions\\
$t\mapsto \int_{-\infty}^{t}\Gamma(t,s)P(s)\, ds,$ $t\in {\mathbb R}$ and $t\mapsto \int^{+\infty}_{t}\Gamma(t,s)P(s)\, ds,$ $t\in {\mathbb R}$ are almost periodic; see, e.g., the proof of \cite[Theorem 2.3]{ding-prim}.
\item[(ii)] We write $\int^{+\infty}_{-\infty}\Gamma(t,s)f(s)\, ds=\int^{+\infty}_{t}\Gamma(t,s)f(s)\, ds+\int^{t}_{-\infty}\Gamma(t,s)f(s)\, ds,$ $t\in {\mathbb R}.$ The both addends can be considered similarly, while the second addend is identically equal to $\int_{-\infty}^{0}\Gamma(t,t+s)f(t+s)\, ds,$ $t\in {\mathbb R}.$
\end{itemize}

\begin{thm}\label{viagra-boys}
Suppose that $a>0,$ $\alpha>0,$ $1\leq p<+\infty,$ $\alpha p \geq 1$ and $ap\geq 1.$
If the function $f : {\mathbb R} \rightarrow X$ is bounded and
$f\in e-(x^{\alpha},t^{-a})-B^{p}({\mathbb R}: X)$,
then the function $u(\cdot)$, given by \eqref{berlinstrase}, is bounded, continuous and belongs to the class $e-(x^{\alpha},t^{-a})-B^{p}({\mathbb R}: X).$
\end{thm}

3. The use of Theorem \ref{krompo} is almost mandatory in the analysis of the existence and uniqueness of Besicovitch almost periodic type solutions for some classes of the abstract semilinear Cauchy problems. The first part of this result has a serious unpleasant drawback because we must impose that the function $G : {\mathbb R}^{n} \times Y \rightarrow Z$ from its formulation is Bohr ${\mathcal B}$-almost periodic, which automatically leads to the existence and uniqueness of almost periodic solutions of the abstract semilinear Cauchy problems under our consideration, in a certain sense. Here we will present the following illustrative application of Theorem \ref{krompo}(ii), with $\zeta=1,$ and Proposition \ref{stan}.

Suppose that the operator family $(R(t))_{t>0}\subseteq L(X)$ satisfies \eqref{rad}, as well as that $1\leq p,\ q<+\infty,$ $1/p+1/q=1,$ $q(\beta-1)>-1$ if $p>1$, and $\beta=1$ if $p=1.$ 
If $(R(t))_{t>0}$ is a solution operator family which governs solutions of the abstract fractional Cauchy inclusion
$
D_{t,+}^{\gamma}u(t)\in {\mathcal A}u(t)+g(t),$ $ t\in {\mathbb R},$ 
where $\gamma \in (0,1),$ and a closed multivalued linear operator ${\mathcal A}$ satisfies condition \cite[(P)]{nova-mono}, then it is usually said that a continuous function $t\mapsto u(t),$ $t\in {\mathbb R}$ is a mild solution of the abstract semilinear fractional Cauchy inclusion 
$$
\text{(SCP):} \ \ \,
D_{t,+}^{\gamma}u(t)\in {\mathcal A}u(t) +G(t;u(t)),\quad t\in {\mathbb R}
$$ 
if and only if 
\begin{align*}
u(t)=\int^{t}_{-\infty}R(t-s)G(s;u(s))\, ds,\quad t\in {\mathbb R}.
\end{align*}
Suppose that $G\in e-({\mathcal B},x,t^{-n/p})-B^{p}_{a,1}({\mathbb R}\times X : X),$ where
$
{\mathcal B}
$ denotes the family of all bounded subsets of $X.$ Suppose, further, that there exists a finite real constant $a>0$ such that \eqref{pravilo} holds with $\alpha=1$, $a\int^{\infty}_{0}\| R(s)\|\, ds<1,$
and $\sup_{t\in {\mathbb R},x\in B}\|G(t;x)\|<+\infty$ for every bounded subset $B$ of $X.$ It can be simply proved that the vector space $C_{b}({\mathbb R} : X) \cap e-(x,t^{-1/p})-B^{p}({\mathbb R} : X)$ equipped with the sup-norm is a Banach space.
Applying Proposition \ref{stan}(ii) and Theorem \ref{krompo}, we get that the mapping $\Phi : C_{b}({\mathbb R} : X) \cap e-(x,t^{-1/p})-B^{p}({\mathbb R} : X) \rightarrow  C_{b}({\mathbb R} : X) \cap e-(x,t^{-1/p})-B^{p}({\mathbb R} : X),$ 
given by $(\Phi u)(t):=\int^{t}_{-\infty}R(t-s)G(s;u(s))\, ds,$ 
$t\in {\mathbb R},$ is a well-defined contraction; therefore, there exists a unique mild solution $u(\cdot)$ of the abstract semilinear inclusion (SCP) which belongs to the space $C_{b}({\mathbb R} : X) \cap e-(x,t^{-1/p})-B^{p}({\mathbb R} : X).$ 
\vspace{0.1cm}

4. Our results about the invariance of Besicovitch-$p$-almost periodicity under the actions of infinite convolution products can be also formulated for the usual convolution 
\begin{align}\label{usual}
f\mapsto F(x)\equiv \int_{{\mathbb R}^{n}}h(x-y)f(y)\, dy,\quad x\in {\mathbb R}^{n},
\end{align}
provided that the function $h\in L^{1}({\mathbb R}^{n})$ has a certain growth order. Before stating a general result in this direction, we will first consider the inhomogeneous heat equation in ${\mathbb R}^{n}$ whose solutions are governed by the action of 
Gaussian semigroup 
\begin{align}\label{gauss}
F\mapsto (G(t)F)(x)\equiv \bigl(4\pi t\bigr)^{-n/2}\int_{{\mathbb R}^{n}}e^{-|y|^{2}/4t}F(x-y)\, dy,\quad t>0,\ x\in {\mathbb R}^{n}.
\end{align}
Without going into full details, we will only note that the formula \eqref{gauss} makes sense even if the function $F(\cdot)$ is polynomially bounded; our basic assumption will be that there exist two finite real numbers $b\geq 0$ and $c>0$ such that $|F(x)|\leq c(1+|x|)^{b},$ $x\in {\mathbb R}^{n}$ as well as that $a>0,$ $\alpha>0,$
$1\leq p<+\infty, $
$\alpha p\geq 1,$ $1/(\alpha p)+1/q=1$
and $F\in e-(x^{\alpha},t^{-a})-B^{p}({\mathbb R}^{n} : {\mathbb C}).$ 
Let us fix a real number $t_{0}$ in \eqref{gauss}. Then the mapping $x\mapsto (G(t_{0})F)(x),$ $x\in {\mathbb R}^{n}$ is well-defined and has the same growth as the inhomogeneity $f(\cdot)$. Now we will prove that belongs to the class $ e-(x^{\alpha},t^{-a})-B^{p}({\mathbb R}^{n} : {\mathbb C})$  as well. Let $\epsilon>0$, let 
$$
c_{t_{0}}:=\bigl(4\pi t_{0}\bigr)^{-n/2}\Bigl\| e^{-|\cdot|^{2}/8t_{0}}\Bigr\|_{L^{q}({\mathbb R}^{n})}^{\alpha p},
$$
and let $\epsilon_{0}>0$ be such that
$$
\epsilon_{0}\cdot 2^{ap}c_{t_{0}}\int_{{\mathbb R}^{n}}e^{-|y|^{2}p/8t_{0}}\bigl( 1+|y|\bigr)^{n}\, dy< \epsilon.
$$
We know that there exist a trigonometric polynomial $P(\cdot)$ and a finite real number $t_{1}>0$ such that
$$
\int_{[-t,t]^{n}}\bigl| F(x)-P(x) \bigr|^{\alpha p}\, dx<\epsilon_{0}t^{ap},\quad t\geq t_{1}.
$$
Furthermore, we know that the function $x\mapsto (G(t_{0})P)(x),$ $x\in {\mathbb R}^{n}$ is Bohr almost periodic (see \cite[Subsection 6.1.7]{nova-selected}) so that the final conclusion simply follows from the next computation (see also the computation from the proof of Proposition \ref{stan}):
\begin{align}
\notag \frac{1}{t^{ap}}&\int_{[-t,t]^{n}}\bigl| (G(t_{0})F)(x)-(G(t_{0})P)(x) \bigr|^{\alpha p}\, dx
\\\label{dojaja}& \leq \frac{c_{t_{0}}}{t^{ap}}\int_{[-t,t]^{n}}\int_{{\mathbb R}^{n}} e^{-|y|^{2}\alpha p/8t_{0}}\bigl| F(x-y)-P(x-y) \bigr|^{\alpha p}\, dy\, dx
\\\notag & =\frac{c_{t_{0}}}{t^{ap}}\int_{{\mathbb R}^{n}}e^{-|y|^{2}\alpha p/8t_{0}}\int_{[-t,t]^{n}} \bigl| F(x-y)-P(x-y) \bigr|^{\alpha p}\, dx\, dy
\\\notag & \leq  \frac{c_{t_{0}}}{t^{ap}}\int_{{\mathbb R}^{n}}e^{-|y|^{2}\alpha p/8t_{0}}\int_{[-t+|y|,t+|y|]^{n}} \bigl| F(x)-P(x) \bigr|^{\alpha p}\, dx\, dy
\\\notag & \leq \frac{c_{t_{0}}}{t^{ap}}\int_{{\mathbb R}^{n}}e^{-|y|^{2}\alpha p/8t_{0}}\epsilon_{0}2^{ap}\bigl(t^{ap}+|y|^{ap} \bigr)\, dy,\quad t\geq t_{1}.
\end{align}
The estimate \eqref{dojaja} is obtained by writing the term $e^{-|y|^{2}/4t_{0}}=e^{-|y|^{2}/8t_{0}}e^{-|y|^{2}/8t_{0}}$ and applying the H\"older inequality after that. This can be also done in the general case; arguing so, we can prove the following result:

\begin{thm}\label{start}
Suppose that $b \geq 0,$ $\alpha>0,$ $a>0,$
$1\leq p<+\infty, $
$\alpha p\geq 1$, $1/(\alpha p)+1/q=1,$
$f\in e-(x^{\alpha},t^{-a})-B^{p}({\mathbb R}^{n} : Y)$ 
and
$ \|f(x)\|_{Y}\leq c(1+|x|)^{b},$ $x\in {\mathbb R}^{n}.$ If there exist two functions $h_{1}: {\mathbb R}^{n} \rightarrow {\mathbb C}$ and $h_{2}: {\mathbb R}^{n} \rightarrow {\mathbb C}$ such that $h=h_{1}h_{2}$, $h_{1}\in L^{q}({\mathbb R}^{n})$ and $|h_{1}(\cdot)|^{\alpha} [1+|\cdot|]^{\zeta}\in L^{p}({\mathbb R}^{n})$ with $\zeta=\max(b\alpha,a),$
then the function $F(\cdot),$ given by \eqref{usual}, belongs to the class $e-(x^{\alpha},t^{-a})-B^{p}({\mathbb R}^{n} : Y)$ and has the same growth order as $f(\cdot).$ 
\end{thm}

We can simply apply Theorem \ref{krompo} and Theorem \ref{start} 
in the analysis of the existence and uniqueness of bounded Besicovitch-$p$-almost periodic solutions for a class of the semilinear Hammerstein integral equations of convolution type on ${\mathbb R}^{n};$ see \cite[p. 362]{nova-selected} for more details. 
It would be very tempting to incorporate Theorem \ref{start} in the analysis of the existence and uniqueness of Besicovitch-$p$-almost periodic solutions of the abstract ill-posed Cauchy problems whose solutions are governed by integrated solution operator families or $C$-regularized solution operator families (cf. \cite[pp. 543--545]{nova-selected} for more details). \vspace{0.1cm}

In the remaining applications, we will consider the usual case $\phi(x)\equiv x$ and the class of Besicovitch-$p$-almost periodic functions.\vspace{0.1cm}

5. Without going into full details, we would like to note that the argumentation contained in our analysis of \cite[Example 3, p. XXXV]{nova-selected} enables one to consider the existence and uniqueness of Besicovitch-$p$-almost periodic type solutions of the wave equation in ${\mathbb R}^{2}$ whose solutions are given by the famous d'Alembert formula. For example, if the functions $f(\cdot)$ and $g^{[1]}(\cdot)$ from this example are Besicovitch-$p$-almost periodic in ${\mathbb R}$, then the solution $u(x,t)$ will be Besicovitch-$p$-almost periodic in ${\mathbb R}^{2}.$ A similar conclusion can be clarified for the solutions of the wave equation given by the Kirchhoff (Poisson) formula; see \cite{nova-selected} for more details. \vspace{0.1cm}

6. In this issue, we continue our analysis of the evolution systems considered in the final application of \cite[Section 6.3, pp. 426--428]{nova-selected}. Suppose that $Y:=L^{r}({\mathbb R}^{n})$ for some $r\in [1,\infty)$  and $ A(t):= \Delta +a(t){\rm I}$, $t \geq 0$, where $\Delta$ is the Dirichlet Laplacian on $L^{r}(\mathbb{R}^{n}),$  and $ a \in L^{\infty}([0,\infty)) $. Then the evolution system $(U(t,s))_{t\geq s\geq 0}\subseteq L(Y)$ generated by the family $(A(t))_{t\geq 0}$ exists; this evolution system is given by $U(t,t):=I$ for all $t\geq 0,$ and
$$ 
[U(t,s)F]({\bf u}):=\int_{{\mathbb R}^{n}} K(t,s,{\bf u},{\bf v})F({\bf v}) \, d{\bf v}, \quad F\in L^{r}({\mathbb R}^{n}),\quad t> s\geq 0,
$$
where 
$$   
K(t,s,{\bf u},{\bf v}):= \bigl(4\pi (t-s)\bigr)^{-\frac{n}{2}} e^{\int_{s}^{t} a(\tau)\, d\tau }\exp \Biggl(-\frac{| x-y|^{2}}{4(t-s)}\Biggr),\quad 
t>s,\ {\bf u},\ {\bf v} \in \mathbb{R}^{n} .
$$
We know that, for every $ \tau \in \mathbb{R}^{n} $, we have  
$$
 K(t,s,{\bf u}+\tau,{\bf v}+\tau)=K(t,s,{\bf u},{\bf v}),\quad   t>s\geq 0,\ {\bf u},\ {\bf v} \in \mathbb{R}^{n} ,
$$
as well as that, under certain assumptions, a unique mild solution of the abstract Cauchy problem 
$
(\partial /\partial t)u(t,x) = A(t)u(t,x) ,$ $t > 0;$ $u(0,x) = F(x)$
is given by
$
u(t,x):=[U(t,0)F](x),$ $t\geq 0,$ $x\in {\mathbb R}^{n}.$
Suppose now that $F(\cdot)\in Y$ is Besicovitch-$p$-almost periodic for some finite exponent $p\geq 1.$ If $\epsilon>0$ is given in advance, then we can find a finite real number
$t_{0}>0$ and a trigonometric polynomial $P(\cdot)$ such that
$$
\int_{[-t,t]^{n}}\bigl| F({\bf u})-P({\bf u}) \bigr|^{p}\, d{\bf u}<\epsilon t^{n},\quad t\geq t_{0}.
$$
The function $
u_{P}(t,x):=[U(t,0)P](x),$ $t\geq 0,$ $x\in {\mathbb R}^{n}$ is well-defined, continuous and satisfies that, if $\tau \in {\mathbb R}^{n}$ is an $\epsilon$-almost period of $P(\cdot)$, then
\begin{align*}
& |u_{P}(t,{\bf u}+\tau)-u_{P}(t,{\bf u})|=\Biggl| \int_{{\mathbb R}^{n}}\bigl[ K(t,0,{\bf u}+\tau,{\bf v}) -K(t,0,{\bf u},{\bf v})\bigr] P({\bf v})\, d{\bf v} \Biggr|
\\& =\Biggl| \int_{{\mathbb R}^{n}}K(t,0,{\bf u}+\tau,{\bf v}+\tau) P({\bf v}+\tau)\, d{\bf v} -\int_{{\mathbb R}^{n}}K(t,0,{\bf u},{\bf v}) P({\bf v})\, d{\bf v} \Biggr|
\\& =\Biggl| \int_{{\mathbb R}^{n}}K(t,0,{\bf u},{\bf v}) \bigl[P({\bf v}+\tau)\, d{\bf v} - P({\bf v})\bigr]\, d{\bf v} \Biggr|
\\& \leq c_{t}\int_{{\mathbb R}^{n}}e^{-\frac{|{\bf u}-{\bf v}|^{2}}{4t}}|P({\bf v}+\tau)-P({\bf v})|\, d{\bf v}
\\& \leq c_{t}\epsilon \int_{{\mathbb R}^{n}}e^{-\frac{|{\bf u}-{\bf v}|^{2}}{4t}}\, d{\bf v}=c_{t}\epsilon \int_{{\mathbb R}^{n}}e^{-\frac{|{\bf v}|^{2}}{4t}}\, d{\bf v},\quad t>0,\ {\bf u}\in {\mathbb R}^{n};
\end{align*}
hence, the function $u_{P}(t,\cdot)$ is almost periodic for every fixed real number $t>0.$ Writing
$$
\int_{{\mathbb R}^{n}}K(t,0,{\bf u},{\bf v}) \bigl[F({\bf v})- P({\bf v})\bigr]\, d{\bf v} =\int_{{\mathbb R}^{n}}K(t,0,{\bf u},{\bf v}-{\bf u}) \bigl[F({\bf v}-{\bf u})- P({\bf v}-{\bf u})\bigr]\, d{\bf v},
$$
and the term $e^{-|{\bf v}|^{2}/4t}$ as $e^{-|{\bf v}|^{2}/8t}\cdot e^{-|{\bf v}|^{2}/8t}$ in the corresponding computation after that, we may conclude as before that the function $u(t,\cdot)$ is Besicovitch-$p$-almost periodic for every fixed real number $t>0.$\vspace{0.1cm}

7. In this part, we will present certain applications of Proposition \ref{sdffds} in the analysis of the existence and uniqueness of Besicovitch almost periodic solutions for certain classes of PDEs; see, e.g., \cite{salsa} and references cited therein. We will  revisit here the classical theories of quasi-linear partial differential equations of first order and the linear partial differential equations of second order with constant coefficients, considering solutions defined on certain proper subsets $\Lambda$ of ${\mathbb R}^{n},$ where $n\geq 2.$\vspace{0.1cm}

7.1. It is well known that the general solution of equation $u_{x}+u_{y}=u$ is given by $u(x,y)=g(y-x)e^{x},$ $(x,y)\in {\mathbb R}^{2},$ where $g :  {\mathbb R} \rightarrow {\mathbb R}$ is a continuously differentiable function. Suppose that $\Lambda :=(-\infty,0]\times {\mathbb R}$ and the function $g(\cdot)$ is Besicovitch-$q$-almost periodic for some finite exponent $q>1.$ Then there exists a sequence $(P_{k})$ of trigonometric polynomials such that
$$
\lim_{k\rightarrow +\infty}\limsup_{t\rightarrow +\infty}t^{-1}\int_{-t}^{t}\bigl|P_{k}(s)-g(s)\bigr|^{q}\, ds=0.
$$ 
Since $Q_{k}(x,y):=P_{k}(x-y),$ $(x,y)\in {\mathbb R}^{2}$ is a sequence of trigonometric polynomials of two variables, the above simply implies along with an elementary argumentation that
$$
\lim_{k\rightarrow +\infty}\limsup_{t\rightarrow +\infty} t^{-2}\int_{-t}^{t}\int_{-t}^{t}\bigl|Q_{k}(x,y)-g(x-y)\bigr|^{q}\, dx\, dy=0,
$$
so that the function $G(x,y):=g(x-y),$ $(x,y)\in {\mathbb R}^{2}$ is Besicovitch-$q$-almost periodic. Since the function $(x,y)\mapsto e^{x},$ $(x,y)\in \Lambda$  is Besicovitch-$p$-almost periodic for any finite exponent $p\geq 1,$ an application of Proposition \ref{sdffds} yields that the solution $u(x,y)$ is Besicovitch-$r$-almost periodic on $\Lambda,$ for any exponent $r\in [1,q).$\vspace{0.1cm}

7.2. Consider the linear partial differential equation of second order with constant coefficients:
\begin{align}\label{cdcd}
Au_{xx}+2Bu_{xy}+Cu_{yy}+2Du_{x}+2Eu_{y}+Fu=0,
\end{align}
where $A,\ B, \ C, \ D, E,\ F$ are real constants such that $B>0,$ $C>0,$
$B^{2} \geq AC,$ $E^{2}\geq CF,$ $B^{2}> E^{2}-CF,$ and
$$
\bigl(BE-CD\bigr)^{2}=\bigl(B^{2}-AC\bigr) \bigl( E^{2}-CF\bigr).
$$
As proposed by J. D. Ke\v cki\' c in \cite{jovan} (see also \cite{jovan1}), the general solution $u(x,y)$ of the equation \eqref{cdcd} is given by
\begin{align*}
u(x,y)&=e^{\Bigl(-\frac{B}{C}+\frac{\sqrt{E^{2}-CF}}{C}\Bigr)y}f\Biggl(x+\Bigl(-\frac{E}{C}+\frac{1}{C}\sqrt{B^{2}-AC} \Bigr)y\Biggr)
\\& +e^{\Bigl(-\frac{B}{C}-\frac{\sqrt{E^{2}-CF}}{C}\Bigr)y}g\Biggl(x+\Bigl(-\frac{E}{C}-\frac{1}{C}\sqrt{B^{2}-AC} \Bigr)y\Biggr),
\end{align*}
where $f : {\mathbb R} \rightarrow {\mathbb R}$ and $g : {\mathbb R} \rightarrow {\mathbb R}$ are arbitrary two times continuously differentiable functions.
Suppose that $\Lambda={\mathbb R} \times [0,+\infty)$ and the functions $f(\cdot)$ and $g(\cdot)$ are Besicovitch-$q$-almost periodic for some finite exponent $q> 1.$ Arguing as above, an application of Proposition \ref{sdffds} yields that the solution $u(x,y)$ of \eqref{cdcd} is Besicovitch-$r$-almost periodic on $\Lambda,$ for any exponent $r\in [1,q).$

We close this section with the observation that it is not clear how one can prove that the solutions considered above are Besicovitch-$q$-almost periodic.

\section{Conclusions and final remarks}\label{konk}

In this paper, we have considered 
various classes of multi-dimensional Besicovitch almost periodic type functions.
The main structural properties of multi-dimensional Besicovitch almost periodic type
functions are established. Many structural results known in the one-dimensional setting are reconsidered and extended to the functions of the form $F: \Lambda \times X \rightarrow Y,$ where $\emptyset \neq \Lambda \subseteq {\mathbb R}^{n},$ $X$ and $Y$ are complex Banach spaces. We have also furnished many illustrative examples, useful remarks and applications to
the abstract Volterra integro-differential equations.

Concerning certain drawbacks of this work, we would like to note that we have not considered here the differentiation and integration of multi-dimensional Besicovitch almost periodic type functions as well as the dual spaces of  multi-dimensional Besicovitch almost periodic type functions.
Besides many other topics, in this paper we have not reconsidered the notions of ${\mathcal M}^{p}$-almost periodicity and  ${\mathcal M}^{p}$-regularity in the multi-dimensional setting as well (see \cite{deda} and \cite{berta} for more details about the subject). The class of metrical multi-dimensional Besicovitch almost periodic functions will be considered somewhere else following the general approach obeyed in \cite{metrical}.

The class of
multi-dimensional $\rho$-almost periodic type functions, extending the class of multi-dimensional $c$-almost periodic type functions when $\rho=c{\rm I},$
have recently been investigated by M. Fe\v ckan et al. \cite{rho}; see also \cite{nova-mse}-\cite{brno}. 
The class of Besicovitch-$(p,c)$-almost periodic functions, where $1\leq p <\infty$ and  $c\neq 1$, has not been analyzed in the existing literature so far, even in the one-dimensional setting. At the end of paper, we will briefly explain that the method proposed in the proof of \cite[Proposition 2.8]{c1}, a fundamental result about $c$-almost periodic functions saying that any $c$-almost periodic function $f: {\mathbb R}\rightarrow Y$ is almost periodic, cannot be essentially employed in the analysis of corresponding classes of
Besicovitch-$(p,c)$-almost periodic type functions:

\begin{example}\label{brasa}
Suppose that $p\geq 1,$ $c\in {\mathbb C}\setminus \{0\},$ $m\in {\mathbb N},$ $c^{m}=1,$ and the function $f : {\mathbb R} \rightarrow {\mathbb C}
$ is Besicovitch-$(p,c)$-almost periodic, i.e., for every $\epsilon>0$ there exists a satisfactorily uniform set $A=\{ \tau_{i} : i\in {\mathbb Z}\}\subseteq {\mathbb R}$ such that
\begin{align}\label{jednac}
\limsup_{t\rightarrow +\infty}\Biggl( \frac{1}{2t}\int^{t}_{-t}\bigl| f(s+\tau_{i})-cf(s)\bigr|^{p}\, ds\Biggr)^{1/p}<\epsilon
\end{align}
and, for every $l>0,$ 
\begin{align}\label{jednac2}
\limsup_{t\rightarrow +\infty}\Biggl( \frac{1}{2t}\int^{t}_{-t}\Biggl[\limsup_{k\rightarrow +\infty}\frac{1}{2k+1}\sum_{i=-k}^{k}l^{-1}\int^{x+l}_{x}\bigl| f(s+\tau_{i})-cf(s)\bigr|^{p}\, ds \Biggr]\, dx\Biggr)^{1/p}<\epsilon;
\end{align}
cf. \cite{besik} for the notion of a satisfactorily uniform set in ${\mathbb R}.$
Then the set $mA$ is also satisfactorily uniform and for each $i\in {\mathbb Z}$ we have:
\begin{align*}
\int^{t}_{-t}&\bigl| f(s+m\tau_{i})-f(s)\bigr|^{p}\, ds
\\& \leq \int_{-t}^{t}\Biggl(\sum_{j=0}^{m-1}|c|^{j} \bigl| f(s+(m-j)\tau_{i})-cf(s+(m-j-1)\tau_{i}) \bigr|\Biggr)^{p}\, ds
\\& \leq c_{m,p}\sum_{j=0}^{m-1}\int^{t+(m-j-1)\tau_{i}}_{-t+(m-j-1)\tau_{i}}\bigl| f(s+\tau_{i})-cf(s) \bigr|^{p}\, ds
\\& \leq 2mc_{m,p}\epsilon^{p}\Bigl(t+\bigl| t+\tau_{i} \bigr|+...+|t+(m-1)\tau_{i}|\Bigr)
\\& \leq 2mc_{m,p}\epsilon^{p}\Bigl( mt+\frac{m^{2}}{2}|\tau_{i}| \Bigr)\leq \mbox{Const.} \cdot 2t\epsilon^{p},\quad t\geq t_{0}(\epsilon,i),
\end{align*}
for some finite real constants $c_{m,p}>0$ and $ t_{0}(\epsilon,i)>0.$ Therefore, \eqref{jednac} holds with the number $c$ replaced by the number $1$ therein, which can be simply transferred to the multi-dimensional setting. But, it is not clear how we can deduce the corresponding conclusion for the equation \eqref{jednac2}; arguing as above, we can only prove that there exists a finite real constant $c_{m,p}>0$ such that:
\begin{align*}
\int^{t}_{-t}&\Biggl[\limsup_{k\rightarrow +\infty}\frac{1}{2k+1}\sum_{i=-k}^{k}l^{-1}\int^{x+l}_{x}\bigl| f(s+\tau_{i})-f(s)\bigr|^{p}\, ds \Biggr]\, dx
\\&  \leq c_{m,p}\Biggl\{ \int^{t}_{-t}\Biggl[\limsup_{k\rightarrow +\infty}\frac{1}{2k+1}\sum_{i=-k}^{k}l^{-1}\int^{x+l}_{x}\bigl| f(s+\tau_{i})-cf(s)\bigr|^{p}\, ds \Biggr]\, dx
\\& +\int^{t}_{-t}\Biggl[\limsup_{k\rightarrow +\infty}\frac{1}{2k+1}\sum_{i=-k}^{k}l^{-1}\int^{x+\tau_{i}+l}_{x+\tau_{i}}\bigl| f(s+\tau_{i})-cf(s)\bigr|^{p}\, ds \Biggr]\, dx
\\& +...+\int^{t}_{-t}\Biggl[\limsup_{k\rightarrow +\infty}\frac{1}{2k+1}\sum_{i=-k}^{k}l^{-1}\int^{x+(m-1)\tau_{i}+l}_{x+(m-1)\tau_{i}}\bigl| f(s+\tau_{i})-cf(s)\bigr|^{p}\, ds \Biggr]\, dx\Biggr\},
\end{align*}
which seems to be completely inapplicable for our purposes. 
\end{example}

To the best knowledge of the author, the class of Besicovitch-$p$-almost periodic functions in ${\mathbb R}^{n}$ has not yet been considered in a Bohr like manner ($n\geq 2$).
For further information concerning Besicovitch almost periodic  functions in ${\mathbb R}^{n}$ and satisfactorily uniform sets in ${\mathbb R}^{n}$, we also refer the reader to the forthcoming research monograph \cite{advances}.

\end{document}